\documentclass[11pt,reqno]{amsart}

\theoremstyle{plain}
\newtheorem{theorem} {Theorem}[section]
\newtheorem{lemma}[theorem] {Lemma}
\newtheorem{proposition}[theorem] {Proposition}

\theoremstyle{definition}
\newtheorem{definition}[theorem] {Definition}

\newtheorem{example} [theorem]{Example}

\theoremstyle{remark}
\newtheorem{remark}[theorem] {Remark}

\numberwithin{equation}{section}

\usepackage{amsfonts}
\usepackage{amssymb}
\usepackage{amscd}
\usepackage{bigints}

\newcommand{\R}{{\mathbb R}}
\newcommand{\Z}{{\mathbb Z}}
\newcommand{\N}{{\mathbb N}}
\newcommand{\NN}{{\mathcal N}}

\newcommand{\PP}{{\mathcal P}}

\newcommand{\T}{{\mathbb T}}

\newcommand{\CC}{{\mathbb C}}

\newcommand{\TT}{{\mathcal T}}

\newcommand{\BB}{{\mathfrak B}}
\newcommand{\al}{{\alpha}}
\newcommand{\la}{{\lambda}}
\newcommand{\sa}{{\sigma}}

\newcommand{\iy}{{\infty}}
\newcommand{\vphi}{{\varphi}}
\newcommand{\vep}{{\varepsilon}}
\newcommand{\g}{{\gamma}}
\newcommand{\de}{{\delta}}

\newcommand{\be}{{\beta}}

\newcommand{\bna}{\begin{eqnarray}}
\newcommand{\ena}{\end{eqnarray}}
\newcommand{\ba}{\begin{eqnarray*}}
\newcommand{\ea}{\end{eqnarray*}}
\newcommand{\beq}{\begin{equation}}
\newcommand{\eeq}{\end{equation}}
\DeclareMathOperator*{\esssup}{ess\,sup}

\begin{document}

\title[Constants in Multivariate Inequalities]
{Sharp Constants of Approximation Theory. VI.
Weighted Polynomial Inequalities
 of Different Metrics on the Multidimensional Cube and Ball }
\author{Michael I. Ganzburg}
 \address{212 Woodburn Drive, Hampton,
 VA 23664\\USA}
 \email{michael.ganzburg@gmail.com}

 \begin{abstract}
 We prove
limit equalities between the sharp constants in weighted Nikolskii-type
inequalities
 for multivariate polynomials on an $m$-dimensional cube and ball
  and the corresponding constants for
 entire functions of exponential type.
 \end{abstract}
 \maketitle

 \section{Introduction}\label{S1}
\setcounter{equation}{0}
\noindent
We continue the study of the sharp constants in multivariate
 inequalities
of approximation theory
that began in \cite{G2019}--\cite{G2020}.
In this paper we prove
asymptotic equalities
between the sharp constants in the
multivariate weighted
 Nikolskii-type inequalities
  for entire functions of exponential type and
algebraic polynomials on  an $m$-dimensional cube and ball.
\vspace{.12in}\\
\textbf{Notation.}
Let $\R^m$ be the Euclidean $m$-dimensional space with elements
$x=(x_1,\ldots,x_m),\, y=(y_1,\ldots,y_m),
\,t=(t_1,\ldots,t_m),\,s=(s_1,\ldots,s_m),\,
v=(v_1,\ldots,v_m)$,
the inner product $(t,x):=\sum_{j=1}^mt_jx_j$,
and the norm $\vert x\vert:=\sqrt{(x,x)}$.
Next, $\CC^m:=\R^m+i\R^m$ is the $m$-dimensional complex
space with elements
$z=(z_1,\ldots, z_m)=x+iy$
and the norm $\vert z\vert:=\sqrt{\vert x\vert^2+\vert y\vert^2}$;
$\Z^m$ denotes the set of all integral lattice points in $\R^m$;
$\Z^m_+$ is a subset of $\Z^m$
of all points with nonnegative coordinates; and
$\N:=\{1,\,2,\ldots\}$.
We also use multi-indices $k=(k_1,\ldots,k_m)\in \Z^m_+$
with
$
 \langle k\rangle:=\sum_{j=1}^m k_j$ and
 $x^k:=x_1^{k_1}\cdot\cdot\cdot x_m^{k_m}.
 $

Given $\sa=(\sa_1,\ldots,\sa_m)\in\R^m\setminus\{0\}$, let
$\Pi^m(\sa):=\{t\in\R^m: \left\vert t_j\right\vert
\le \left\vert\sa_j\right\vert,
 1\le j\le m\}$
 be the $l$-dimensional parallelepiped in $\R^m$,
where $l\ge 1$ is the number of nonzero coordinates of
$\sa$.
Given $M>0$, let
$Q^m(M):=\{t\in\R^m: \left\vert t_j\right\vert\le M, 1\le j\le m\},\,
\BB^m(M):=\{t\in\R^m: \left\vert t\right\vert\le M\},\,
O^m(M):=\{t\in\R^m: \sum_{j=1}^m\left\vert t_j\right\vert\le M\}$, and
$S^{m-1}:=\{t\in\R^m:\vert t\vert=1\}$ be
the $m$-dimensional
 cube, ball,
octahedron, and the $(m-1)$-dimensional unit sphere in $\R^m$,
 respectively.
Next, let $Q^m:=Q^m(1)$ and $\BB^m:=\BB^m(1)$.
In addition, $\vert \Omega\vert_l$ denotes the $l$-dimensional
 Lebesgue measure
of a  measurable set $\Omega\subseteq\R^m,\,1\le l\le m$.
We set $S^0:=\{-1,1\}$ and $\left\vert S^0\right\vert_0:=2$.
We also use the floor function
 $\lfloor a \rfloor$, the gamma function $\Gamma(z)$,
 and the beta function $B(z,w)$.

Let $W:\Omega\to [0,\iy]$ be a locally integrable weight
 on a measurable
 subset $\Omega$ of $\R^m$, and let
 $L_{r,W}(\Omega)$ be a weighted space of all measurable
  complex-valued
  functions $F:\Omega\to \CC$
  with the finite quasinorm
 \ba
 \|F\|_{L_{r,W}(\Omega)}=
 \|F\|_{L_{r,W(x)}(\Omega)}:=\left\{\begin{array}{ll}
 \left(\int_\Omega\vert F(x)\vert^r W(x) dx\right)^{1/r}, & 0<r<\iy,\\
 \esssup_{x\in \Omega} \vert F(x)\vert, &r=\iy.
 \end{array}\right.
 \ea
 In the nonweighted case ($W =1$) and also in the case of $r=\iy$, we set
\ba
&&\|\cdot \|_{L_r (\Omega)} : =\|\cdot \|_{L_{r,1} (\Omega)} ,
\quad  L_r (\Omega) := L_{r,1}
(\Omega),\qquad 0<r< \infty ,\\
&&\|\cdot\|_{L_{\iy}(\Omega)}:= \|\cdot\|_{L_{\iy,W}(\Omega)},\qquad
  L_\iy (\Omega) := L_{\iy,W}(\Omega).
\ea
The quasinorm $\|\cdot\|_{L_{r,W}(\Omega)}$
allows the following "triangle" inequality:
\beq\label{E1.1}
\left\| F+G\right\|^{\tilde{r}}_{L_{r,W}(\Omega)}
 \le  \left\|F\right\|^{\tilde{r}}_{L_{r,W}(\Omega)}
 +\left\|G\right\|^{\tilde{r}}_{L_{r,W}(\Omega)},
 \qquad F\in L_{r,W}(\Omega),\quad G\in L_{r,W}(\Omega),
 \eeq
 where $\tilde{r}:=\min\{1,r\}$ for $r\in(0,\iy]$.

 In this paper we will need certain definitions and properties of
 convex bodies in $\R^m$.
Throughout the paper $V$ is a centrally symmetric (with respect to the origin)
 closed
 convex body in $\R^m$ and
 $V^*:=\{y\in\R^m: \forall\, t\in V, \vert (t,y)\vert \le 1\}$
 is the \emph{polar} of $V$.
 It is well known that $V^*$ is a centrally symmetric (with respect to the origin)
 closed
 convex body in $\R^m$ and $V^{**} =V$ (see, e.g., \cite[Sect. 14]{R1970}).
 The set $V$ generates the dual norm
 on $\CC^m$ by
 $
  \|z\|_V^*:=\sup_{t\in V}\left\vert\sum_{j=1}^m t_jz_j\right\vert,\,z\in\CC^m.
 $

 \begin{definition}\label{D1.0}
 A body $V\subset \R^m$ satisfies
 the \emph{parallelepiped condition ($\Pi$-condition)} if
 for every vector $t\in V\setminus\{0\}$
 the parallepiped $\Pi^m(t)$ is a subset of $V$.
 \end{definition}
 \noindent
 It is easy to verify that $V$ satisfies the $\Pi$-condition if and only if
 $V$ is symmetric about all coordinate hyperplanes, that is,
 for every $t\in V$ the vectors $(\pm\vert t_1\vert,\ldots,
 \pm\vert t_m\vert)$
 belong to $V$.
 A slightly different version of the $\Pi$-condition,
 which is equivalent to Definition \ref{D1.0},
 was introduced in \cite[Sect. 1]{G2020a}.
   In particular, given $\la\in [1,\iy]$ and
   $\sa\in\R^m,\,\sa_j>0,\,1\le j\le m$,
   the set $V_{\la,\sa}:=\left\{t\in\R^m:
 \left(\sum_{j=1}^m\left\vert t_j/\sa_j\right\vert^{\la}
 \right)^{1/\la}\le 1\right\}$,
 satisfies the $\Pi$-condition.
 Therefore, the sets $\Pi^m(\sa)$ (for $\la=\iy$),
 $Q^m(M)$ (for $\la=\iy$ and $\sa=(M,\ldots,M)$),
 $\BB^m(M)$ (for $\la=2$ and $\sa=(M,\ldots,M)$), and
 $O^m(M)$ (for $\la=1$ and $\sa=(M,\ldots,M)$)
 satisfy the $\Pi$-condition as well.

 Given $a\ge 0$, the set of all trigonometric polynomials
 $T(x)=\sum_{\eta\in aV\cap \Z^m}c_\eta\exp[i(\eta,x)]$
  with complex
 coefficients is denoted by $\TT_{aV}$.
 In the univariate case we
  use the notation $\TT_n:=\TT_{[-a,a]}=\TT_{[-n,n]}$,
   where $n=\lfloor a\rfloor,\,n\in\Z^1_+$.

  \begin{definition}\label{D1.1}
 We say that an entire function $f:\CC^m\to \CC^1$ has exponential type $V$
 if for any $\vep>0$ there exists a constant $C_0(\vep,f)>0$ such that
 for all $z\in \CC^m$,
 $\vert f(z)\vert\le C_0(\vep,f)\exp\left((1+\vep)\|z\|_V^*\right)$.
 \end{definition}
 \noindent
   The set of all entire function of exponential type $V$ is denoted
  by $B_V$, and the set of even functions in each variable from $B_V$
  is denoted by $B_{V,e}$.
  In the univariate case we use the notation
  $B_\la:=B_{[-\la,\la]}$ and $B_{\la,e}:=B_{[-\la,\la],e},\,\la>0$.
  In addition, note that if $V$ satisfies the $\Pi$-condition and
  $f\in B_V$, then the function $f(\pm z_1,\ldots,\pm z_m)$ belongs
   to $B_V$ as well since $\left\|(\pm z_1,\ldots,\pm z_m)\right\|^*_V
   =\left\|z\right\|^*_V$ by the definition of $\left\|z\right\|^*_V,\,
   z\in \CC^m$.

  Throughout the paper, if no confusion may occur,
  the same notation is applied to
  $f\in B_V$ and its restriction to $\R^m$ (e.g., in the form
  $f\in  B_V\cap L_{p,W}(\R^m))$.
  The class $B_V$ was defined by Stein and Weiss
  \cite[Sect. 3.4]{SW1971}. For $V=\Pi^m(\sa),\,V=Q^m(M),$ and
  $V=\BB^m(M)$, similar
  classes were
  defined by Bernstein \cite{B1948} and Nikolskii
   \cite[Sects. 3.1, 3.2.6]{N1969}, see also
  \cite[Definition 5.1]{DP2010}. In particular, $B_{\BB^m(M)}$ is
  the set of all entire functions of spherical type $M$
  (see \cite[Sect. 3.2.6]{N1969}).
  Certain properties of functions from $B_V$ are presented
   in Lemma \ref{L3.2}.

  Given a convex  bounded set $\Omega\subset\R^m$,
  let  $\PP_{\Omega}$ be the set of all
  polynomials $P(x)=\sum_{k \in \Omega\cap \Z_+^m}c_k x^k$
   in $m$ variables
   with
  complex coefficients
  whose Newton polyhedra
  (see, e.g., \cite[Sect. 3]{AVGK1984} for the definition)
   are subsets of $\Omega$.
  In this paper we use the set $\PP_{aV}$ for
  a given $a\ge 0$ and the set
  $\PP_{aV,e}$ of even polynomials in each variable
  from $\PP_{aV}$.
  In the case of $V=O^m(1),\,\PP_{nV}=\PP_{O^m(n)}$
  coincides with the set
  $\PP_{n,m}$
   of all polynomials
   $P(x)=\sum_{\langle k\rangle\le n}c_k x^k$
   in $m$ variables of total degree at most $n,\,n\in\Z^1_+$.
   In the univariate case we
  use the notation $\PP_n:=\PP_{[-a,a]}=\PP_{n,1}$, where
  $n=\lfloor a\rfloor,\,n\in\Z^1_+$.
   Newton polyhedra and polynomial classes $\PP_{\Omega}$
    associated with
Newton polyhedra play an important role in algebra, geometry,
 and analysis
(see, e.g., a survey \cite[Sect. 3]{AVGK1984}).
Note that if $V$ satisfies the $\Pi$-condition, then
the Newton polyhedron of a polynomial from $\PP_{nV}$
is a downward closed set (see, e.g., \cite[Sect. 2]{M2015}).
  It is easy to verify that if $V_1\subseteq V_2$, then
  $B_{V_1}\subseteq B_{V_2}$ and $\PP_{aV_1}\subseteq\PP_{aV_2}$.

  Throughout the paper $A_1,\,A_2,\,C,\,C_1,\ldots,C_{19}$
  denote positive constants independent
of essential parameters.
 Occasionally we indicate dependence on certain parameters.
 The same symbol $C$ does not
 necessarily denote the same constant in different occurrences,
 while $A_1,\,A_2$, and $C_l,\,1\le l\le 19$,
 denote the same constants in different occurrences.
  \vspace{.12in}\\
\textbf{Inequalities of Different Metrics.}
 We first define  sharp constants in Nikolskii-type
inequalities for algebraic and trigonometric polynomials
and  entire functions of exponential type
and then briefly discuss their asymptotic behaviours.

Let
$x_0\in \Omega$ be a fixed point
and let $B$ be a subspace of $L_{p,W}(\Omega)$,
where $\Omega$ is a closed subset of $\R^m$.
In the capacity of $x_0$ we shall use either the origin $0$ or
boundary points, and
in the capacity of $B$ we shall use sets of
algebraic and trigonometric polynomials
and entire functions of exponential type.
 Let us define two sharp constants of different metrics
\bna
&&\NN_{x_0}\left(B,L_{p,W}(\Omega)\right)
:=\sup_{h\in B\setminus\{0\}}\frac{\vert h(x_0)\vert}
{\|h\|_{L_{p,W}(\Omega)}},\label{E1.2}\\
&&\NN\left(B,L_{p,W}(\Omega)\right)
:=\sup_{h\in B\setminus\{0\}}\frac{\|h\|_{L_{\iy}(\Omega)}}
{\|h\|_{L_{p,W}(\Omega)}}.\label{E1.3}
\ena
These constants usually coincide
for all $x_0\in\Omega$ in case of invariant subspaces $B,\,W=1$,
and homogeneous spaces $\Omega$ like $\R^m$,
the torus $\T^m$, and the sphere $S^{m-1}$.
However, it is not the case in other situations,
in particular, for algebraic polynomials.
Our major goal here is to find the asymptotic behaviour
of $\NN\left(B,L_{p,W}(\Omega)\right)$ in certain cases of
multivariate polynomial subspaces $B$, sets $\Omega$, and
Gegenbauer-type weights $W$.

The following limit relations for multivariate trigonometric polynomials
\bna\label{E1.5}
&&\lim_{a\to\iy}a^{-m/p}\NN\left(\TT_{aV},L_{p}(Q^m(\pi))\right)
=\lim_{a\to\iy}a^{-m/p}\NN_0\left(\TT_{aV},L_{p}(Q^m(\pi))\right)\nonumber\\
&&=\NN\left(B_{V}\cap L_{p}(\R^m),L_{p}(\R^m)\right)
=\NN_0\left(B_{V}\cap L_{p}(\R^m),L_{p}(\R^m)\right),
\qquad p\in(0,\iy),
\ena
were proved by the author \cite[Theorem 1.3]{G2018}.
In the univariate case of $V=[-1,1]$ and $a\in\N$,
 \eqref{E1.5}  were proved by the author and Tikhonov \cite{GT2017}.
In  earlier publications \cite{LL2015a, LL2015b}, Levin and Lubinsky established
versions of \eqref{E1.5} on the unit circle.
Quantitative estimates of the remainder in asymptotic equalities of the
Levin-Lubinsky type were found by Gorbachev and Martyanov
\cite[Corollary 1]{GM2020}.
Certain extensions of the Levin-Lubinsky's results to the $m$-dimensional
unit sphere in $\R^{m+1}$ were recently proved by Dai, Gorbachev, and Tikhonov
\cite{DGT2018}
(see also \cite[Corollary 4.5]{G2019}).

The classic inequalities of different metrics for algebraic polynomials
\beq\label{E1.6}
\| P\|_{L_\iy(\Omega)}\le C\,a^\mu \|P\|_{L_p(\Omega)}, \qquad P\in \PP_{aV},
\quad p\in[1,\iy),
\eeq
where $\Omega\subset\R^m$ is a bounded closed domain
and $C$ is independent of $P$ and $a$,
have been studied since the 1960s.
The exponent $\mu$ in \eqref{E1.6} essentially depends on $\Omega$.
For example, inequality in \eqref{E1.6}
 holds true for $\mu=2m/p$ and any domain $\Omega$,
 satisfying the cone condition, in particular for
 convex bodies (not necessarily symmetric),
 see  Daugavet \cite[Theorem 1]{D1972}, \cite[Theorem 2]{D1976}
 and the author \cite[Theorem 2]{G1978}. It is also valid for
 $\mu=(m+1)/p$ and any domain $\Omega$ with the smooth boundary,
 see Daugavet \cite[Theorem 2]{D1972}, \cite[Theorem 5]{D1976}
 and Kro\'{o} and Schmidt \cite[p. 433]{KS1997}. In case of the
  cube $\Omega=Q^m$ for $\mu=2m/p$  and the unit ball
   $\Omega=\BB^m$
  for $\mu=(m+1)/p$, the factor $a^\mu$ in \eqref{E1.6} cannot be replaced
  by a smaller one as $a\to\iy$ (see \cite[Theorem 7]{D1976}).
  More examples and discussions are
  presented by Ditzian and Prymak \cite{DP2010, DP2016}.
  In addition, note that weighted versions of \eqref{E1.6}
  with Gegenbauer-type weights were obtained in
  \cite[Theorems 1 and 2]{D1972}, \cite[Theorems 2 and 5]{D1976}
  and with $k$-concave weights in \cite[Theorem 2.3]{G2016}.

  If $\| P\|_{L_\iy(\Omega)}$ in \eqref{E1.6} is replaced
  by $\vert P(x_0)\vert,\,
  x_0\in\Omega$, then
  unlike Nikolskii-type inequalities for trigonometric polynomials
   (cf. \eqref{E1.5}),
  the exponent $\mu$ in \eqref{E1.6} also depends on the location of $x_0$.
  In particular, if $x_0$ is an interior point of $V$, then $\mu=m/p$
  (see \cite[Lemma 2.9]{G2020a}).
  The asymptotic behaviours of the corresponding sharp constants
   for $x_0=0$ were found by the author in the following forms:
  \bna
  &&\lim_{a\to\iy}a^{-m/p}\NN_0\left(\PP_{aV},L_{p}(Q^m)\right)
  =\NN_0\left(B_V\cap L_{p}(\R^m),L_{p}(\R^m)\right),
  \qquad p\in(0,\iy),\label{E1.6a}\\
  &&\lim_{n\to\iy}n^{-m/p}\NN_0\left(\PP_{n,m},L_{p}(V^*)\right)
  =\NN_0\left(B_V\cap L_{p}(\R^m),L_{p}(\R^m)\right),
  \qquad p\in(0,\iy).\label{E1.6b}
  \ena
Relation  \eqref{E1.6a}
for $V$, satisfying the $\Pi$-condition,
 was proved in \cite[Theorem 1.2]{G2020a}, and
\eqref{E1.6b} was established in \cite[Theorem 1.2]{G2019b}
(see also \cite[Theorem 1.1]{G2017} for $m=1$ and
\cite[Corollary 4.4]{G2019} for $V=\BB^m$ and $p\in[1,\iy)$).

The exact value of the sharp constant
$\NN\left(\PP_{n},
  L_{p,W}([-1,1])\right)$ is known in several cases.
  Geronimus \cite[Theorem 1]{Ge1938} found
 $\NN\left(\PP_{n},
  L_{1,(1-u^2)^{-1/2}}([-1,1])\right)$.
   Simonov and Glazyrina \cite[Theorem 1]{SG2015}
   generalized this result by using and developing the
   Geronimus method.
   Note that the constants found in  \cite{Ge1938}
   and \cite{SG2015} are not explicit, and the problem
   of finding their asymptotic behaviour as $n\to\iy$
   is still open.
For certain weights, the sharp constant
$\NN\left(\PP_{n},
  L_{2,W}([-1,1])\right)$
  can be found by using extremal properties of orthonormal
  polynomials.
Using this approach, Lupas \cite{L1974}
(see also \cite[Theorem 6.1.8.2]{MMR1994})
 found $\NN\left(\PP_{n},
  L_{2,W}([-1,1])\right)$ for the Jacobi weight.
  For the Gegenbauer weight, his result can be reduced
  to the following formula ($\la\ge 0$):
  \bna\label{E1.6ba}
  \NN\left(\PP_{n},
  L_{2,(1-u^2)^{\la-1/2}}([-1,1])\right)
  &=&\left(\frac{(2\la+2n+1)\Gamma(2\la+n+1)}
  {2^{2\la} (2\la+1)\Gamma^2(\la+1/2) n!}\right)^{1/2}
  \ena
  (see \cite{L1974} and \cite[Eq. 6.1.8.8]{MMR1994}).
  Two special cases of \eqref{E1.6ba} are well-known.
  The first of them,
  $\NN\left(\PP_{n},
  L_{2}([-1,1])\right)=2^{-1/2}(n+1)$,
  has been known since the 1920s
  (see Polya and
  Szeg\"{o} \cite[Problem 6.103]{PS1998} and
  Labelle \cite{L1969}).
  The second one is
  \ba
  \NN\left(\PP_{n},
  L_{2,(1-u^2)^{-1/2}}([-1,1])\right)
  =2^{1/2}\NN\left(\TT_{n},
  L_{2}([-\pi,\pi])\right)
  =((2n+1)/\pi)^{1/2}
  \ea
  (see \cite[Sect. 4.9.2]{T1963}).

  In addition, the following univariate
  relations are known
  ($p\in [1,\iy),\,\la\ge 0$):
  \bna\label{E1.6c}
  &&\lim_{n\to\iy}
  n^{-(2\la+1)/p}\NN\left(\PP_{n},
  L_{p,(1-u^2)^{\la-1/2}}([-1,1])\right)
  \nonumber\\
  &&=\lim_{n\to\iy}
  n^{-(2\la+1)/p}\NN_1\left(\PP_{n},
  L_{p,(1-u^2)^{\la-1/2}}([-1,1])\right)
  \nonumber\\
  &&=2^{1/p}\NN_0\left(B_{1,e}\cap L_{p,\vert u\vert^{2\la}}(\R^1),
  L_{p,\vert u\vert^{2\la}}(\R^1)\right)\nonumber\\
  &&=2^{1/p}\NN_0\left(B_{1}\cap L_{p,\vert u\vert^{2\la}}(\R^1),
  L_{p,\vert u\vert^{2\la}}(\R^1)\right).
  \ena
  In addition to \eqref{E1.6c}, there exists a function
 $g_{0}\in  \left(B_{1,e}\cap L_{p,\vert u
 \vert^{2\la}}(\R^1)\right)\setminus\{0\}, p\in[1,\iy),$ such that
 \beq \label{E1.6ca}
  \lim_{n\to\iy}n^{-\left(2\la+1\right)/p}\NN\left(\PP_{n},
  L_{p,\left(1-u^2\right)^{\la-1/2}}\left([-1,1]\right)\right)
= 2^{1/p} \vert g_{0}(0)\vert/\|g_{0}
\|_{ L_{p,\vert u\vert^{2\la}}(\R^1)}.
 \eeq
  The first relation in \eqref{E1.6c} immediately follows
  from the equality
  \beq\label{E1.6d}
  \NN\left(\PP_{n},L_{p,(1-u^2)^{\la-1/2}}([-1,1])\right)
  =\NN_1\left(\PP_{n},L_{p,(1-u^2)^{\la-1/2}}([-1,1])\right)
  \eeq
  proved by Arestov and Deikalova \cite[Theorem 1]{AD2015}
   (see also Example \ref{Ex2.5} of Section \ref{S2n}),
  while the second one in \eqref{E1.6c} and also \eqref{E1.6ca} were
   proved in \cite[Theorem 4.3]{G2019}.
  The third equality in \eqref{E1.6c} immediately follows from the
  following well-known symmetrization:
  if $f\in B_{1}\cap L_{p,\vert u\vert^{2\la}}(\R^1),\,p\in[1,\iy)$,
  then the function $f^*(u):=(f(u)+f(-u))/2$ satisfies the properties
  \ba
  f^*\in B_{1,e}\cap L_{p,\vert u\vert^{2\la}}(\R^1),
  \qquad
  f^*(0)=f(0),\qquad
  \|f^*\|_{ L_{p,\vert u\vert^{2\la}}(\R^1)}
  \le \|f\|_{ L_{p,\vert u\vert^{2\la}}(\R^1)}.
  \ea
The asymptotic formula
   for the univariate sharp constant in \eqref{E1.6}
   with $\Omega=[-1,1]$ and $p\in[1,\iy)$
  is a consequence of \eqref{E1.6c}
  in the following form:
  \beq\label{E1.7}
   \lim_{n\to\iy}n^{-2/p}
   \NN\left(\PP_{n},L_{p}([-1,1])\right)
   =2^{1/p}\NN_0\left(B_{1}\cap L_{p,\vert u\vert}(\R^1),
   L_{p,\vert u\vert}(\R^1)\right)
   \eeq
   (see \cite[Corollary 4.6]{G2019}).
  A different version of  \eqref{E1.7}
   was proved in \cite[Theorem 1.4]{G2017} (see also
\cite[p. 94]{G2017}).

Note that relations \eqref{E1.5}, \eqref{E1.6a}, \eqref{E1.6b},
 and the second relation in \eqref{E1.6c} for
 the sharp  Nikolskii-type constants
  are special cases of more general limit relations between
 sharp   Markov--Bernstein--Nikolskii type constants
  with $\vert h(0)\vert$ and $\|h\|_{L_\iy(\Omega)}$
  in \eqref{E1.2} and \eqref{E1.3} replaced by
  $\vert D_N(h)(0)\vert$ and $\|D_N(h)\|_{L_\iy(\Omega)}$,
  respectively,
  for certain differential operators $D_N$
  (see \cite{GT2017, G2017, G2018, G2019b, G2020a}).
In addition, note that the sharp
nonweighted Bernstein--Nikolskii  constants
for entire functions of exponential type
can be easily found only for $p=2$
(see \cite[Eq. (1.6)]{G2018}).

No exact or asymptotic equalities for
 multivariate sharp constants
 in Nikolskii-type inequalities
 for algebraic polynomials are known.
In this paper we extend \eqref{E1.6c} to
asymptotic relations for
multivariate sharp constants
$\NN\left(\PP_{aV},L_{p,W}(Q^m)\right)$ and
 $\NN\left(\PP_{n,m},L_{p,W}(\BB^m)\right)$
 with Gegenbauer-type weights.
 In addition, we prove a weighted version
 of limit equality
 \eqref{E1.6a} and find certain multivariate
  sharp constants for $p=2$.
  Note that we are unaware of any applications of these
  asymptotically sharp results.
\vspace{.12in}\\
 \textbf{Main Results and Remarks.} We first discuss
 asymptotics of the sharp constants on the cube $Q^m$.

 \begin{theorem} \label{T1.2}
 If $V\subset\R^m$ satisfies the $\Pi$-condition, then for
 $p\in[1,\iy),\,m\ge 1,$
 and $\la_j\ge 0,\,1\le j\le m$,
 the following limit relation holds true:
 \bna \label{E1.8}
  &&\lim_{a\to\iy}a^{-\left(m+2\sum_{j=1}^m\la_j\right)/p}
  \NN\left(\PP_{aV},
  L_{p,\prod_{j=1}^m\left(1-x_j^2\right)^{\la_j-1/2}}
  (Q^m)\right)\nonumber\\
  &&=2^{m/p}\,\NN_0\left(B_{V}\cap
  L_{p,\prod_{j=1}^m\left\vert t_j\right\vert^{2\la_j}}(\R^m),
  L_{p,\prod_{j=1}^m\left\vert t_j\right\vert^{2\la_j}}(\R^m)\right).
 \ena
 In addition, there exists a function
 $f_0\in  \left(B_V\cap L_{p,\prod_{j=1}^m\left\vert t_j
 \right\vert^{2\la_j}}(\R^m)\right)\setminus\{0\}$
  such that
 \bna \label{E1.9}
  &&\lim_{a\to\iy}a^{-\left(m+2\sum_{j=1}^m\la_j\right)/p}
  \NN\left(\PP_{aV},
  L_{p,\prod_{j=1}^m\left(1-x_j^2\right)^{\la_j-1/2}}(Q^m)\right)\nonumber\\
&&= 2^{m/p}\left\vert f_0(0)\right\vert/\|f_0\|_
{L_{p,\prod_{j=1}^m\left\vert t_j\right\vert^{2\la_j}}(\R^m)}.
 \ena
 \end{theorem}
 \noindent
 One of the major ingredients of the proof of Theorem
 \ref{T1.2}
 is the following theorem of independent interest
  that presents a weighted version of relation
 \eqref{E1.6a}.
 \begin{theorem} \label{T1.2a}
 If $V\subset\R^m$ satisfies the $\Pi$-condition, then for
 $p\in (0,\iy),\,m\ge 1,$
 and $\al_j\ge 0,\,\be_j\ge -1/2,\,1\le j\le m$,
 the following limit relation holds true:
 \bna \label{E1.8a}
  &&\lim_{a\to\iy}a^{-\left(m+\sum_{j=1}^m\al_j\right)/p}
  \NN_{0}\left(\PP_{aV},
  L_{p,\prod_{j=1}^m\left\vert t_j\right\vert^{\al_j}
  \left(1-t_j^2\right)^{\be_j}}(Q^m)\right)\nonumber\\
  &&=\NN_0\left(B_{V}\cap L_{p,\prod_{j=1}^m
  \left\vert t_j\right\vert^{\al_j}}(\R^m),
  L_{p,\prod_{j=1}^m\left\vert t_j\right\vert^{\al_j}}(\R^m)\right).
 \ena
 In addition, there exists a function
 $f_0\in  \left(B_V\cap L_{p,\prod_{j=1}^m\left\vert t_j
 \right\vert^{\al_j}}(\R^m)\right)\setminus\{0\}$
  such that
 \bna \label{E1.9a}
  &&\lim_{a\to\iy}a^{-\left(m+\sum_{j=1}^m\al_j\right)/p}
  \NN_0\left(\PP_{aV},
  L_{p,\prod_{j=1}^m\left\vert t_j\right\vert^{\al_j}\left(1-t_j^2
  \right)^{\be_j}}(Q^m)\right)\nonumber\\
&&= \vert f_0(0)\vert/\|f_0\|_
{L_{p,\prod_{j=1}^m\left\vert t_j\right\vert^{\al_j}}(\R^m)}.
 \ena
 \end{theorem}
 \noindent
 Next, the asymptotic of the sharp constant on the unit ball $\BB^m$
is discussed.
 \begin{theorem} \label{T1.3}
 For  $p\in[1,\iy),\,m\ge 1,$
 and $\la \ge 0$,
 the following equalities hold true:
 \bna \label{E1.10}
  &&\lim_{n\to\iy}n^{-\left(m+2\la\right)/p}
  \NN\left(\PP_{n,m},
  L_{p,\left(1-\vert x\vert^2\right)^{\la-1/2}}\left(\BB^m\right)\right)\nonumber\\
  &&=A_1\,\NN_0\left(B_{1,e}\cap L_{p,\vert u\vert^{m+2\la-1}}(\R^1),
  L_{p,\vert u\vert^{m+2\la-1}}(\R^1)\right)\nonumber\\
  &&=A_2\NN_0\left(B_{\BB^m}\cap L_{p,\vert t\vert^{2\la}}(\R^m),
  L_{p,\vert t\vert^{2\la}}(\R^m)\right),
 \ena
 where
  \bna
  &&A_1=A_1(m,p,\la):=\left(\frac{2\Gamma(\la+m/2)}
  {\pi^{(m-1)/2}\Gamma(\la+1/2)}\right)^{1/p},\label{E1.10a}\\
  &&A_2=A_2(m,p,\la):=\left(\frac{2\pi^{1/2}\Gamma(\la+m/2)}
  {\Gamma(m/2)\Gamma(\la+1/2)}\right)^{1/p}.\label{E1.10ab}
  \ena
 In addition, there exists a function
 $f_0\in  \left(B_{\BB^m}\cap L_{p,\left\vert t
 \right\vert^{2\la}}(\R^m)\right)\setminus\{0\}$ such that
 \beq \label{E1.10b}
  \lim_{n\to\iy}n^{-\left(m+2\la\right)/p}\NN\left(\PP_{n,m},
  L_{p,\left(1-\vert x\vert^2\right)^{\la-1/2}}
  \left(\BB^m\right)\right)
= A_2 \vert f_0(0)\vert/\|f_0\|_{ L_{p,\vert t\vert^{2\la}}(\R^m)}.
 \eeq
 \end{theorem}

 Finally, the sharp constants in \eqref{E1.10} can be found for $p=2$.
 \begin{theorem} \label{T1.3a}
 For  $m\ge 1$
 and $\la \ge 0$,
 the following equalities hold true:
 \bna
 &&\NN\left(\PP_{n,m},
  L_{2,\left(1-\vert x\vert^2\right)^{\la-1/2}}
  \left(\BB^m\right)\right)
  \nonumber\\
  &&=\left(\frac{(2\la+2n+m)\Gamma(2\la+n+m)}
  {2^{2\la+m-1} \pi^{(m-1)/2}(2\la+m)\Gamma(\la+1/2)
   \Gamma(\la+m/2)n!}\right)^{1/2};\label{E1.10c}\\
   &&\NN_0\left(B_{\BB^m}\cap L_{2,\vert t\vert^{2\la}}(\R^m),
  L_{2,\vert t\vert^{2\la}}(\R^m)\right)\nonumber\\
  &&=\left(\frac{\Gamma(m/2)}
  {2^{2\la+m-1} \pi^{m/2}(2\la+m)
   \Gamma^2(\la+m/2)}\right)^{1/2}.\label{E1.10d}
  \ena
  \end{theorem}

 \begin{remark}\label{R1.4}
Relations \eqref{E1.8a} and \eqref{E1.9a} show that the function
$f_0\in B_{V}\cap L_{p,\prod_{j=1}^m
\left\vert t_j\right\vert^{\al_j}}(\R^m)$
from Theorem \ref{T1.2a} is an extremal function for
$\NN_0\left(B_{V}\cap
  L_{p,\prod_{j=1}^m\left\vert t_j\right\vert^{\al_j}}(\R^m),
  L_{p,\prod_{j=1}^m\left\vert t_j\right\vert^{\al_j}}(\R^m)\right)$.
  Moreover, the extremal function $f_0$ in \eqref{E1.9a}
   can be chosen from $B_{V,e}\cap L_{p,\prod_{j=1}^m
\left\vert t_j\right\vert^{\al_j}}(\R^m)$ for $p\in[1,\iy)$.
This result can be proved
 by the following symmetrization trick for $p\in[1,\iy)$:
  the function
 \ba
  f_{0,e}(t):=2^{-m}\sum_{\left\vert \de_j\right\vert=1,\,1\le j\le m}
 f_{0}\left(\de_1t_1,\ldots, \de_mt_m\right)
 \ea
 belongs to $B_{V,e}$ and
 \ba
 \frac{\left\vert f_0(0)\right\vert}
 {\left\| f_0\right\|_{L_{p,\prod_{j=1}^m\left\vert t_j\right
  \vert^{\al_j}}(\R^m)}}
  \le \frac{\left\vert f_{0,e}(0)\right\vert}
 {\left\| f_{0,e}\right\|_{L_{p,\prod_{j=1}^m\left\vert t_j\right
  \vert^{\al_j}}(\R^m)}}.
 \ea
In addition, for $p\in[1,\iy),\, \PP_{aV}$ in \eqref{E1.8a}
 and \eqref{E1.9a}
 can be replaced by $\PP_{aV,e}$, and also
 $B_{V}$ in \eqref{E1.8} and \eqref{E1.8a}
 and $B_{\BB^m}$ in \eqref{E1.10} and \eqref{E1.10d}
 can be replaced by $B_{V,e}$ and $B_{\BB^m,e}$, respectively.
 These results can be proved by
 the same symmetrization trick;
 see also general symmetrization (invariance) theorems in
 \cite[Theorems 1.1 and 1.2]{G2019}.
 \end{remark}

 \begin{remark}\label{R1.4a}
  Limit relations
  \eqref{E1.8} and \eqref{E1.10}
  coincide for $m=1$,
  and their slightly different versions
 were established in \cite[Theorem 4.3]{G2019}
 (see also \eqref{E1.6c}).
 A special case of Theorem \ref{T1.2a} for $\al_j=\be_j=0,\,1\le j\le m$,
was proved in \cite[Theorem 1.2]{G2020a} (see also \eqref{E1.6a}).
  Note that to prove Theorem \ref{T1.2}, it suffices to prove \eqref{E1.8a} for
 $p\in[1,\iy)$
 and $\al_j=2\la_j,\,\be_j=\la_j-1/2,\,\la_j\ge 0,\,1\le j\le m$,
 but it is possible to prove Theorem \ref{T1.2a} for the wider range of
 $p\in(0,\iy)$ and $\al_j\ge 0,\,\be_j\ge -1/2,\,1\le j\le m$,
  without additional steps.
  In addition, note that equality \eqref{E1.10c} for $m=1$ is reduced to
  \eqref{E1.6ba}
  and equality \eqref{E1.10d} is known for $\la=0$
  (see, e.g., \cite[Eq. (1.6)]{G2018}).
 \end{remark}

\begin{remark}\label{R1.5}
In definitions \eqref{E1.2} and \eqref{E1.3} of the sharp constants, we
discuss only complex-valued functions $P$ and $f$. We can define similarly
the "real" sharp constants if the suprema in \eqref{E1.2} and \eqref{E1.3}
 are taken over all real-valued functions
on $\R^m$ from $\PP_{aV}\setminus\{0\}$
and $(B_V\cap L_{p,W}(\R^m))\setminus\{0\}$, respectively.
It turns out that the "complex" and "real" sharp constants coincide.
 For $m=1$ this fact was proved in \cite[Sect. 1]{G2017} (cf.
\cite[Theorem 1.1]{GT2017} and \cite[Remark 1.5]{G2019b}),
and the case of $m>1$ can be proved similarly.
\end{remark}

\begin{remark}\label{R1.5a}
Note that the problem of finding the asymptotic
behaviour of
$\NN\left(B,L_{p,W}(\Omega)\right)$ for $B,\,W$, and $\Omega$
other than those ones in Theorems \ref {T1.2}, \ref {T1.3},
and \ref {T1.3a} is still open.
In addition, the following two open questions are raised by
an anonymous referee:
is it possible firstly, to replace  the Gegenbauer-type
weights of Theorems \ref {T1.2} and \ref {T1.2a} by
more general Jacobi-type weights
and secondly, to prove a multivariate version of \eqref{E1.7}
 for the cube $Q^m$ and $p\in(0,1)$ as it was done
in \cite[Sect. 3.1]{GM2020}?
\end{remark}
The proofs of Theorems \ref {T1.2}--\ref {T1.3a}
 are presented in Sections \ref{S3n}--\ref{S4n}.
 The proof of Theorem \ref {T1.2a}
 is a weighted version of the proof
 of  \cite[Theorem 1.2]{G2020a},
 and it follows general ideas developed
  in \cite[Corollary 7.1]{G2020}.

To prove Theorems \ref {T1.2} and \ref {T1.3},
  we first
  reduce the corresponding sharp constants
  \eqref{E1.3} of these theorems to constants \eqref{E1.2}.
  This important step is accomplished
  by using multivariate analogues of
equality \eqref{E1.6d} that
are discussed in Section \ref{S2n}.
In particular, Lemma \ref{L2.1n} of
independent interest discusses in general settings
the equality
$\NN\left(B,L_{p,W}(\Omega)\right)
=\NN_{x_0}\left(B,L_{p,W}(\Omega)\right)$
with a special point $x_0$ on the
boundary of $\Omega$.

The next step of the proofs of
Theorems \ref {T1.2} and \ref {T1.3}
(and Theorem \ref {T1.3a} as well)
is based
on equalities between
the sharp constants on the cube $Q^m$
and the ball $B^m$ discussed
in Sections  \ref{S3an} and \ref{S4n}.
In particular, an important role in the proofs of
such the equalities from Section
\ref{S4n} plays Proposition \ref{P4.1} of
independent interest that discusses
the representation of a polynomial that is invariant
under certain rotation subgroups.
Finally, Theorem \ref {T1.2} is reduced to
Theorem \ref {T1.2a} and Theorem \ref {T1.3}
is reduced to
known relations \eqref{E1.6c}.

\section{Extremal Polynomials
and Generalized Translation Operators}
\label{S2n}
\setcounter{equation}{0}
\noindent
Let $\Omega$ be a compact subset of $\R^m$, and let
 $B$ be a finite-dimensional subspace of continuous
functions from $L_{p,W}(\Omega)$ whose elements
are called polynomials. Recall that
$\NN\left(B,L_{p,W}(\Omega)\right)$ is defined by \eqref{E1.3}.

We say that $P^*\in B\setminus\{0\}$
is \emph{an extremal polynomial for}
$\NN\left(B,L_{p,W}(\Omega)\right)$ if
\ba
\NN\left(B,L_{p,W}(\Omega)\right)
=
\frac{\| P^*\|_{L_\iy(\Omega)}}
{\|P^*\|_{L_{p,W}(\Omega)}}.
\ea

In this section we discuss an important property
of certain extremal polynomials for \linebreak
$\NN\left(B,L_{p,W}(\Omega)\right)$
 that the uniform norm of these polynomials
is attained at special points on the boundary of $\Omega$.
We shall use this result for the proofs
of Theorems \ref{T1.2}, \ref{T1.3}, and \ref{T1.3a}
 in the cases of
$B=\PP_{aV},\,a\ge 0,\,\Omega=Q^m$, and
$B=\PP_{n,m},\,n\in\Z^1_+,\,\Omega=\BB^m$.
\vspace{.12in}\\
\textbf{General Result.}
Let $\Omega\subset \R^m ,\,X_0\subseteq \Omega$, and $G\subset \R^l$
($m\in\N,\,l\in\N$) be compact sets.

We assume that there exists a function (operation)
$y=O(t,x):G\times \Omega\to\Omega$, either satisfying
 the \emph{strong surjective condition} (SSC), that is,
for any $y_0\in\Omega$ there exist $t_0\in G$  and $x_0\in X_0$ such that
$O(t_0,x_0)=y_0$,
or  satisfying
 the \emph{very strong surjective condition} (VSSC), that is,
for any $y_0\in\Omega$ and $x_0\in X_0$ there exists $t_0\in G$   such that
$O(t_0,x_0)=y_0$.

We also assume that there exists a family of operators
$\left\{T_t(P,\Omega)\right\}_{t\in G}$,
satisfying the following conditions:
\bna
&&T_t(P,\Omega): B\to B,\qquad t\in G;
\label{E2.1n}\\
&&T_t(P,\Omega)(x_0)=P(O(t,x_0)),\qquad t\in G,\quad x_0\in X_0,
\quad P\in B;\label{E2.2n}\\
&&\left\|T_t(P,\Omega)\right\|_{L_{p,W}(\Omega)}
\le \left\|P\right\|_{L_{p,W}(\Omega)},\qquad t\in G,
\quad P\in B,\quad p\in(0,\iy).\label{E2.3n}
\ena
Various examples of $B,\,\Omega,\,X_0,\,G,\,
O(t,x)$, and $T_t(P,\Omega)$ are discussed below in
Examples \ref{Ex2.2}--\ref{Ex2.6}, \ref{Ex2.8},
and \ref{Ex2.10}.
The following general result is valid:

\begin{lemma}\label{L2.1n}
Let conditions \eqref{E2.1n}, \eqref{E2.2n}, and \eqref{E2.3n}
on $\left\{T_t(P,\Omega)\right\}_{t\in G}$ be satisfied.\\
(a) If an operation $O(t,x)$ satisfies the SSC,
then there exist $x_0\in X_0$ and an extremal polynomial
$P^*\in B\setminus\{0\}$ for $\NN\left(B,L_{p,W}(\Omega)\right)$
such that
$\|P^*\|_{L_\iy(\Omega)}=\vert P^*(x_0)\vert$.\\
(b) If an operation $O(t,x)$ satisfies the VSSC,
then given $x_0\in X_0$,  there exists an extremal polynomial
$P^*\in  B\setminus\{0\}$ for $\NN\left(B,L_{p,W}(\Omega)\right)$
 such that
$\|P^*\|_{L_\iy(\Omega)}=\vert P^*(x_0)\vert$.
\end{lemma}
\begin{proof}
The set $\Omega$ is compact and $B$ is a
finite-dimensional subspace of continuous
functions from $L_{p,W}(\Omega)$.
Hence the existence of an extremal polynomial $P_0$
for $\NN\left(B,L_{p,W}(\Omega)\right)$
 can be proved by
the standard compactness argument. Indeed,
given $d\in\N$, let $P_{d}\in B$
satisfy the following
relations:
$
\|P_{d}\|_{L_\iy\left(\Omega\right)}=1
$
and
\ba
\NN\left(\PP_{aV},L_{p,W}(\Omega)\right)
<\|P_{d}\|_{L_\iy\left(\Omega\right)}
/\|P_{d}\|_{L_p\left(\Omega\right)} +1/d.
\ea
Then there exist a nontrivial polynomial
$P_0\in B$ and a sequence of polynomials
$\{P_{d_l}\}_{l=1}^\iy\subseteq\N$ such that
$\lim_{l\to\iy}P_{d_l}(x)=P_0(x)$
uniformly on $\Omega$. Thus $P_0$ is
an extremal polynomial  for $\NN\left(B,L_{p,W}
(\Omega)\right)$.

Next, assume that
\beq\label{E2.4n}
\|P_{0}\|_{L_\iy\left(\Omega\right)}
=\vert P_0(y_0)\vert,\qquad y_0\in\Omega.
\eeq
If $O(t,x)$ satisfies the SSC and $y_0\notin X_0$,
then there exist $t_0\in G$ and $x_0\in X_0$ such that $O(t_0,x_0)=y_0$.
If $O(t,x)$ satisfies the VSSC and $y_0\ne x_0$ for a given $x_0\in X_0$,
then there exists $t_0\in G$ such that $O(t_0,x_0)=y_0$.

Therefore, the function
$P^*:=T_{t_0}(P_0,\Omega)$ belongs to $B$ by
\eqref{E2.1n}, and $P^*$ satisfies the condition
\beq\label{E2.5n}
\vert P^*(x_0)\vert=\|P_{0}\|_{L_\iy\left(\Omega\right)}
\eeq
by \eqref{E2.2n} and \eqref{E2.4n}.
Furthermore, by \eqref{E2.3n} and \eqref{E2.5n},
${\| P_0\|_{L_\iy(\Omega)}}/
{\|P_0\|_{L_p(\Omega)}}
\le
\vert P^*(x_0)\vert/
{\|P^*\|_{L_p(\Omega)}}.$
Therefore, $P^*$ is
an extremal polynomial for $\NN\left(B,L_{p,W}(\Omega)\right)$ and
$\|P^*\|_{L_\iy(\Omega)}=\vert P^*(x_0)\vert$. Thus the lemma is established.
\end{proof}
\noindent
\textbf{Examples.}
In the capacity of sets $B,\,\Omega,\,X_0,\,G$
and the operation $O(t,x)$
we use in this paper the following objects:
\begin{example}\label{Ex2.2}
$B=\PP_n,\,n\in\Z^1_+,\,
\Omega=[-1,1],\,X_0=\{-1,1\},\,G=[-1,1],\,l=m=1,\,
O(t,x)=tx,\,t\in [-1,1],\,x\in [-1,1]$.
The operation $O(t,x)$ satisfies the VSSC since
given $y_0,\, \left\vert y_0\right\vert\le 1$, and given
$x_0,\, \left\vert x_0\right\vert=1$, one can choose
$t_0=y_0\, \mbox{sgn}\, x_0$.
\end{example}
\begin{example}\label{Ex2.3}
$B=\PP_{aV},\,a\ge 0,\,
\Omega=Q^m,\,X_0=\{x\in Q^m:\vert x_j\vert=1,1\le j\le m\},\,
G=Q^m,\,l=m,\,
O(t,x)=(t_1x_1,\ldots, t_mx_m),\,t\in Q^m,\,x\in Q^m$.
The operation $O(t,x)$ satisfies the VSSC since
given $y_0=\left\{y_{0,1},\ldots,y_{0,m}\right\},\,
 \left\vert y_{0,j}\right\vert\le 1,\,1\le j\le m$
 and given $x_0=\left\{x_{0,1},\ldots,x_{0,m}\right\},\,
 \left\vert x_{0,j}\right\vert= 1,\,1\le j\le m$,
 one can choose
$t_0=\left\{y_{0,1}\, \mbox{sgn}\, x_{0,1},\ldots,
y_{0,m}\, \mbox{sgn}\, x_{0,m}\right\}$.
\end{example}
\begin{example}\label{Ex2.4}
$B=\PP_{n,m},\,n\in\Z^1_+,\,
\Omega=\BB^m,\,X_0=S^{m-1},\,
G=[-1,1],\,l=1,\,
O(t,x)=t x,\,t\in [-1,1],\,x\in \BB^m$.
The operation $O(t,x)$ satisfies the SSC since
given $y_0,\, \left\vert y_0\right\vert\le 1$,
one can choose
$t_0=\pm\left\vert y_0\right\vert$ and
$x_0=\pm y_0/\left\vert y_0\right\vert$, if $y_0\ne 0$, and
$x_0\in S^{m-1}$, if $y_0=0$.
\end{example}
\noindent
\textbf{Generalized Translation Operators.}
In the capacity of the family of operators $T_t(P,\Omega)$ from Lemma
\ref{L2.1n} we use \emph{generalized translation operators} (GTOs).
The GTOs are linear operators of the form
\ba
T_t(f,\Omega)(x)=\int_{\Omega^*}f(\Phi(t,x,s))d\mu(s),
\qquad t\in G\subset \R^l,\quad x\in\Omega\subset\R^m,
\ea
where $\mu$ is a probability measure on
$\Omega^*\subset \R^\nu,\,\nu\in\N$. Note that mostly
$\Omega^*=\Omega$ but there are exceptions (see Examples
\ref{Ex2.5} and \ref{Ex2.8}).
The GTOs $T_t(\cdot,\Omega)$ are generated by product formulae for orthogonal
 polynomials on $\Omega$ with respect to a weighted measure.

 Since the 1960s the numerous GTOs have been defined, studied,
 and applied to univariate and multivariate approximation by
 algebraic polynomials on various domains and surfaces
 (see, e.g., \cite{BBP1968, BSW1980, LN1983, G1998, X2005a, X2005}
 and references therein).
 Applications of the GTOs to analysis of sharp constants in univariate
 inequalities
 of different metrics were initiated by Arestov, Deikalova,
  and Rogozina \cite{DR2012, AD2014, AD2015}.

  Examples of the Gegenbauer-type multivariate GTOs that are needed
  for the proofs of Theorems
  \ref{T1.2}, \ref{T1.3}, and \ref{T1.3a}
  are presented in Examples \ref{Ex2.6},
  \ref{Ex2.8}, and \ref{Ex2.10}. Note that the GTO from
  Example \ref{Ex2.6} is a new one.
  We also discuss the univariate Gegenbauer GTO in Example \ref{Ex2.5}
  to illustrate the transition to multivariate ones.
  This GTO is a special case of more complicated Examples \ref{Ex2.6},
  \ref{Ex2.8}, and \ref{Ex2.10}.

  \begin{example}\label{Ex2.5}
  \textbf{Univariate Gegenbauer GTO.}
  We first define sets
  $\Omega=[-1,1],\,X_0=\{-1,1\},\linebreak
  G=[-1,1]$ and the operation
$O(t,x)=tx,\,t\in [-1,1],\,x\in [-1,1]$, from Example \ref{Ex2.2}
with $O(t,x)$, satisfying the VSSC.
In addition, let $B=\PP_n,\,n\in\Z^1_+$.
Next, given $\la\ge 0$
we define the weight $W(x):=\left(1-x^2\right)^{\la-1/2},\,x\in[-1,1]$,
and the probability measure $\mu_{1,\la}(s)$
 on $[-1,1]$ by the formula
\beq\label{E2.6n}
d\mu_{1,\la}(s):=\left\{\begin{array}{ll}
(1/C_1(\la))\left(1-s^2\right)^{\la-1}ds,& \la>0,\\
(1/2)d(\de_1(s)+\de_{-1}(s)), &\la=0,
\end{array}\right.
\eeq
where
$C_1(\la):=\pi^{1/2}\Gamma(\la)/\Gamma(\la+1/2)$ and
$\de_b$ is the Dirac measure centered at $b\in[-1,1]$.

Furthermore, we define the GTO by the formula
\bna\label{E2.7n}
T_t(f,[-1,1])(x)
&:=&\left\{\begin{array}{ll}
\frac{1}{C_1(\la)}
\int_{-1}^1
f\left(tx+s\sqrt{1-t^2}\sqrt{1-x^2}\right)(1-s^2)^{\la-1}ds,
& \la>0,\\
\frac{1}{2}\left(
f\left(tx+\sqrt{1-t^2}\sqrt{1-x^2}\right)
+f\left(tx-\sqrt{1-t^2}\sqrt{1-x^2}\right)\right),&\la=0,
\end{array}\right.\nonumber\\
&=&\int_{-1}^1
f\left(tx+s\sqrt{1-t^2}\sqrt{1-x^2}\right)d\mu_{1,\la}(s),
\ena
where $t\in[-1,1]$ and $x\in[-1,1]$.
Note that GTO \eqref{E2.7n} is generated by the classic
product formula for the Gegenbauer polynomials $C_n^\la,\,
n\in\Z^1_+,\,\la\ge 0$,
\beq\label{E2.8n}
T_t\left(C_n^\la,[-1,1]\right)(x)=C_n^\la(t)C_n^\la(x)/C_n^\la(1),
\eeq
see \cite[Sect. 11.5]{W1944} for $\la>0$,
while for $\la=0$ product formula \eqref{E2.8n} for the
Chebyshev polynomials is trivial.

Next, $T_t(\cdot,[-1,1]):L_{p,W}([-1,1])\to L_{p,W}([-1,1]),\,
1\le p\le\iy$, is a linear operator whose norm is $1$ for all
 $t\in[-1,1]$ (see \cite[Lemmas 5, 6]{AD2015}).
 This fact immediately implies condition \eqref{E2.3n},
 while condition \eqref{E2.2n} follows from \eqref{E2.7n}.
 In addition, condition \eqref{E2.1n} is satisfied as well
 since by \eqref{E2.8n}, $T_t(P,[-1,1])\in\PP_{n}$ if
 $P\in\PP_{n},\,n\in\Z^1_+$.

 Therefore, applying Lemma \ref{L2.1n} (b),
 we arrive at the existence of an extremal polynomial
 for $\NN
 \left(\PP_n,L_{p,(1-x^2)^{\la-1/2}}([-1,1])\right),
 \,p\in[1,\iy),\,\la\ge 0,$
 whose uniform norm is attained at
 a fixed endpoint of $[-1,1]$.
  This result was obtained by Arestov and Deikalova
  \cite[Theorem 1]{AD2015} by the ingenious trick of
  using GTO \eqref{E2.7n}.
\end{example}

 \begin{example}\label{Ex2.6}
  \textbf{Gegenbauer-type GTO on a Cube.}
  We first define sets
  $\Omega=Q^m,\,X_0=\{x\in Q^m:\vert x_j\vert=1,1\le j\le m\},\,
G=Q^m$ and the operation
$O(t,x)=(t_1x_1,\ldots, t_mx_m),\,t\in Q^m,\,x\in Q^m$,
 from Example \ref{Ex2.3} with $O(t,x)$, satisfying the VSSC.
 In addition, let $B=\PP_{aV},\,a\ge 0$,
 where $V$ satisfies the $\Pi$-condition.
Next, given $\mathbf{\la}=(\la_1,\ldots,\la_m),\,\la_j\ge 0,\,1\le j\le m,$
  we define the weight
$W(x):=\prod_{j=1}^m\left(1-x_j^2\right)^{\la_j-1/2},\,x\in Q^m$,
and the probability measure $\mu_{m,\mathbf{\la}}(s)$ on $Q^m$ by the formula
$d\mu_{m,\mathbf{\la}}(s):=\prod_{j=1}^md\mu_{1,\la_j}(s_j)$,
where $d\mu_{1,\la_j}(s_j),\,1\le j\le m,$ is
defined by \eqref{E2.6n}.

Furthermore, we define the GTO by the formula
\bna\label{E2.9n}
&&T_t(f,Q^m)(x)
:=\prod_{j=1}^mT_{t_j}\left(f\left(\ldots,y_j,\ldots\right),
[-1,1]\right)\left(x_j\right)
\nonumber\\
&&=\int_{Q^m}
f\left(t_1x_1+s_1\sqrt{1-t_1^2}\sqrt{1-x_1^2},\ldots,
t_mx_m+s_m\sqrt{1-t_m^2}\sqrt{1-x_m^2}\right)d\mu_{m,
\boldsymbol{\la}}(s),
\ena
where $t\in Q^m,\,x\in Q^m$,
and $T_{t_j}\left(h\left(y_j\right),
[-1,1]\right)\left(x_j\right)
=T_{t_j}\left(h,[-1,1]\right)\left(x_j\right),\,
1\le j\le m,$ is defined
by \eqref{E2.7n}.
Note that GTO \eqref{E2.9n} is generated by the
product formula for the multivariate Gegenbauer-type
 polynomials
 $\mathbf{C}_\mathbf{n}^\mathbf{\la}(y):=
 \prod_{j=1}^mC_{n_j}^{\la_j}(y_j),\,
n_j\in\Z^1_+,\,\la_j\ge 0,\,1\le j\le m$,
\beq\label{E2.10n}
T_t\left(\mathbf{C}_\mathbf{n}^\mathbf{\la},Q^m\right)(x)
=\mathbf{C}_\mathbf{n}^\mathbf{\la}(t)
\mathbf{C}_\mathbf{n}^\mathbf{\la}(x)/\prod_{j=1}^mC_{n_j}^{\la_j}(1),
\eeq
which follows from \eqref{E2.7n}, \eqref{E2.8n}, and \eqref{E2.9n}.

Next, $T_t(\cdot,Q^m):L_{p,W}(Q^m)\to L_{p,W}(Q^m),\,
1\le p\le\iy$, is a linear operator whose norm is $1$ for all
 $t\in Q^m$. This result follows from the univariate one
 \cite[Lemmas 5, 6]{AD2015}
 and \eqref{E2.9n} by standard induction on $m$.
 Thus condition \eqref{E2.3n} is satisfied,
 while condition \eqref{E2.2n} immediately follows from \eqref{E2.9n}.

 In addition, $T_t(P,Q^m)\in\PP_{aV}$ if
 $P\in\PP_{aV},\,a\ge 0$. Indeed, let
 \ba
 c_k y^k=\sum_{0\le n_j\le k_j,\,1\le j\le m}
 d_{\mathbf{n}}\mathbf{C}_\mathbf{n}^\mathbf{\la}(y),
 \qquad k \in aV\cap\Z^m_+,
 \ea
  be a monomial of $P$. Then using product formula \eqref{E2.10n},
  we see that $T_t(c_k y^k,Q^m)(x)$ is a polynomial
  of degree at most $k_j$ in variable $x_j,\,1\le j\le m$.
  Since the convex set $V$ satisfies the $\Pi$-condition,
  this polynomial belongs to $\PP_{aV}$. Then $T_t(P,Q^m)\in\PP_{aV}$
  because $T_t(\cdot,Q^m)$ is a linear operator.
  Thus condition \eqref{E2.1n} is satisfied.

  Therefore, applying Lemma \ref{L2.1n} (b),
 we arrive at the following lemma:
 \begin{lemma}\label{L2.7}
  There exists an extremal polynomial $P^*$ for
  $\NN\left(\PP_{aV},
  L_{p,\prod_{j=1}^m\left(1-x_j^2\right)^{\la_j-1/2}}(Q^m)\right),
  p\in[1,\iy),\,\la_j\ge 0,\,1\le j\le m,$
   whose uniform norm is attained at
 a fixed vertex $x_0$ of $Q^m$, that is,
 $\|P^*\|_{L_\iy(Q^m)}=\vert P^*(x_0)\vert$ for a given
 $x_0\in X_0$.
\end{lemma}
\end{example}

\begin{example}\label{Ex2.8}
  \textbf{Chebyshev-type GTO on the Unit Ball.}
  We first define sets
  $\Omega=\BB^m,\,X_0=S^{m-1},\,G=[-1,1]$ and the operation
$O(t,x)=tx,\,t\in [-1,1],\,x\in \BB^m$, from Example \ref{Ex2.4}
 with $O(t,x)$, satisfying the SSC.
In addition, let $B=\PP_{n,m},\,n\in\Z^1_+$.
Next,
we define the weight $W(x)
:=\left(1-\vert x\vert^2\right)^{-1/2},\,x\in \BB^m$,
and the probability measure $\mu_{m,0}^*(s)$ on
$S^{m-1}$ by the formula
$
d\mu_{m,0}^*(s):=
(1/C_2(m))ds,
$
where $ds$ is the $(m-1)$-dimensional surface element of
$S^{m-1}$ and
$C_2(m):=\left\vert S^{m-1}\right\vert_{m-1}=2\pi^{m/2}/\Gamma(m/2)$.

Furthermore, we define the GTO by the formula
\beq\label{E2.11n}
T_t(f,\BB^m)(x)
:=\int_{S^{m-1}}
f\left(tx+\sqrt{1-t^2}s D(x)H(x)\right)d\mu_{m,0}^*(s),
\eeq
where $t\in[-1,1]$ and $x\in \BB^m$. Here, $s$ and $x$ are
$m$-dimensional row vectors, $D(x)$ is the $m\times m$ diagonal matrix
whose diagonal elements are $\sqrt{1-\vert x\vert^2},\,1,\ldots, 1$,
and $H(x)$ is an $m\times m$ orthogonal matrix
whose first row is $x/\vert x\vert$.

GTO \eqref{E2.11n} was defined by the author \cite{G1989, G1998}
and it is generated by the product formula for a system of polynomials
(eigenfunctions of $T_t(f,\BB^m)$)
\beq\label{E2.12n}
 \Phi_{l,N}(x):=P_l(x)C_{2N}^{l+(m-1)/2}
 \left(\sqrt{1-\vert x\vert^2}\right),\,
 x\in \BB^m,\,l\in\Z^1_+,\,N\in\Z^1_+,
 \eeq
 in the form
\beq\label{E2.13n}
T_t\left(\Phi_{l,N},\BB^m\right)(x)
=C_n^{(m-1)/2}(t)\Phi_{l,N}(x)/C_n^{(m-1)/2}(1),
\qquad n=l+2N,
\eeq
(see \cite[Lemma 2]{G1998}). Here, $P_l$ is a spherical
harmonic of degree $l,\,l\in\Z^1_+$. In addition, the system
$\{\Phi_{l,N}\}_{l+2N\le n}$ forms an orthogonal basis for
$\PP_{n,m}$ on $\BB^m$ with respect to the weight $W$, that is,
 \beq\label{E2.14n}
 P(x)=\sum_{0\le l+2N\le n}c_{l,N}\Phi_{l,N}(x),\qquad P\in\PP_{n,m},
 \eeq
(see \cite[Lemma 1]{G1998}).

Note that formula \eqref{E2.11n} for the Chebyshev-type GTO on $\BB^m$
follows from the classic formula for the GTO on the sphere $S^{m}$
(see, e.g., \cite{W1981, R1994}) since
 \bna\label{E2.15n}
 T_t(f,\BB^m)(x)=T_t\left(F,S^{m}\right)(x)
 &:=&(C_2(m))^{-1}(1-t^2)^{-(m-1)/2}
 \int_{\left(\tilde{x},\tilde{y}\right)=t}F\left(\tilde{y}\right)
 d\tilde{y}\nonumber\\
 &=&(C_2(m))^{-1}\int_{S_x^{m-1}}
 F\left(t\tilde{x}+\sqrt{1-t^2}\tilde{s}\right) d\tilde{s},
 \ena
 where $\tilde{x}=(x,x_{m+1})\in S^m,\,\tilde{y}=(y,y_{m+1})\in S^m,\,
 \tilde{s}=(s,s_{m+1})\in S^m,\,
 F\left(\tilde{y}\right)=f(y),$ and $S_x^{m-1}$ is the intersection
 of $S^m$ and the equatorial hyperplane orthogonal to the vector
 $\tilde{x}$. Here $x,\,y$, and $s$ are points of $\BB^m$.

 Roughly speaking, the GTO $T_t(f,\BB^m)$ is the "projection"
 of the spherical GTO $T_t\left(F,S^{m}\right)$ on $\BB^m$.
 So all major properties of $T_t(f,\BB^m)$ follow from the corresponding
 well-studied properties of $T_t\left(F,S^{m}\right)$.

 In particular, $T_t(\cdot,\BB^m):L_{p,W}(\BB^m)\to L_{p,W}(\BB^m),\,
1\le p\le\iy$, is a linear operator whose norm is $1$ for all
 $t\in[-1,1]$. Indeed, it is known \cite[Lemma 4.2.2]{BBP1968} that
 for $F\in L_{p}(S^m), 1\le p\le\iy$,
 \beq\label{E2.16n}
 \left\|T_t(F,S^m)\right\|_{L_{p}(S^m)}
 \le \|F\|_{L_{p}(S^m)}.
 \eeq
 Choosing $F(\tilde{y})=f(y),\,\tilde{y}\in S^m, y\in \BB^m$,
 we obtain from \eqref{E2.15n} and \eqref{E2.16n}
 \beq\label{E2.17n}
 \left\|T_t(f,\BB^m)\right\|_{L_{p,W}(\BB^m)}
 =\left\|T_t(F,S^m)\right\|_{L_{p}(S^m)}
 \le \|F\|_{L_{p}(S^m)}=\left\|f\right\|_{L_{p,W}(\BB^m)}.
 \eeq
 Equality in \eqref{E2.17n} holds true for a constant $f$, so
 the norm of $T_t(\cdot,\BB^m)$ is 1.
Then condition \eqref{E2.3n} is satisfied,
 while condition \eqref{E2.2n} immediately follows from \eqref{E2.11n}.

 In addition to \eqref{E2.17n}, note that relations \eqref{E2.13n} and
 \eqref{E2.14n} for polynomials \eqref{E2.12n}
  follow from the corresponding properties of spherical
 harmonics (see \cite[Sect. 2]{LN1983} and
 \cite[Theorem 4.2.1]{SW1971}, respectively).
 Thus by \eqref{E2.13n} and
 \eqref{E2.14n}, $T_t(P,\BB^m)\in\PP_{n,m}$ if
 $P\in\PP_{n,m},\,n\in\Z^1_+$.
 Then condition \eqref{E2.1n} is satisfied.

 Therefore, applying Lemma \ref{L2.1n} (a),
 we arrive at the following lemma:
 \begin{lemma}\label{L2.9}
  There exists an extremal polynomial $P^*$
  for $\NN\left(\PP_{n,m},
  L_{p,\left(1-\vert x\vert^2\right)^{-1/2}}
  (\BB^m)\right),
  \,p\in[1,\iy),$
  whose uniform norm is attained on
 the sphere $S^{m-1}$, that is, there exists
 $x_0\in S^{m-1}$ such that
 $\|P^*\|_{L_\iy(\BB^m)}=\vert P^*(x_0)\vert$.
\end{lemma}
\end{example}

\begin{example}\label{Ex2.10}
  \textbf{Gegenbauer-type GTO on the Unit Ball.}
  We first define sets
  $\Omega=\BB^m,\,X_0=S^{m-1},\,G=[-1,1]$ and the operation
$O(t,x)=tx,\,t\in [-1,1],\,x\in \BB^m$,
   from Example \ref{Ex2.4}  with $O(t,x)$, satisfying the SSC.
   In addition, let $B=\PP_{n,m},\,n\in\Z^1_+$.
Next, given $\la> 0$
we define the weight $W(x):=\left(1-\vert x\vert^2
\right)^{\la-1/2},\,x\in \BB^m$,
and the probability measure $\mu_{m,\la}^*(s)$
 on $\BB^m$ by the formula
$
d\mu_{m,\la}^*(s):=
(1/C_3(m,\la))(1-\vert s\vert^2)^{\la-1}ds,
$
where
$C_3(m,\la):=\pi^{m/2}\Gamma(\la)/\Gamma(\la+m/2)$.

Furthermore, we define the GTO by the formula
\beq\label{E2.18n}
T_t(f,\BB^m)(x)
:=\int_{\BB^m}
f\left(tx+\sqrt{1-t^2}s D(x)H(x)\right)d\mu_{m,\la}^*(s),
\eeq
where $t\in[-1,1]$ and $x\in \BB^m$. Here, $s$ and $x$ are
$m$-dimensional row vectors, and $D(x)$ and $H(x)$ are the same
matrices as in \eqref{E2.11n}.

GTO \eqref{E2.18n} was defined by  Xu \cite[Corollary 3.7]{X2005},
and it is generated by the
product formula for the generalized Gegenbauer polynomials
$C_n^{\la+(m-1)/2}((v,y)),\,
v\in \BB^m,\, y\in S^{m-1},\,n\in\Z^1_+,\,\la> 0$,
\beq\label{E2.19n}
T_t\left(C_n^{\la+(m-1)/2}((\cdot,y)),\BB^m\right)(x)
=C_n^{\la+(m-1)/2}(t)C_n^{\la+(m-1)/2}((x,y))/C_n^{\la+(m-1)/2}(1)
\eeq
for all $y\in S^{m-1}$ (see \cite[p. 500]{X2005}).

Next, $T_t(\cdot,\BB^m):L_{p,W}(\BB^m)\to L_{p,W}(\BB^m),\,
1\le p\le\iy$, is a linear operator whose norm is $1$ for all
 $t\in[-1,1]$ (see \cite[Proposition 3.4(5)]{X2005}).
 This fact immediately implies condition \eqref{E2.3n},
 while condition \eqref{E2.2n} follows from \eqref{E2.18n}.
 In addition, by \eqref{E2.19n}, $T_t(P,\BB^m)\in\PP_{n,m}$ if
 $P\in\PP_{n,m},\,n\in\Z^1_+$
 (see \cite[Proposition 3.4(3)]{X2005}). Then
 condition \eqref{E2.1n} is satisfied.

 Therefore, applying Lemma \ref{L2.1n} (a),
 we arrive at the following lemma:
 \begin{lemma}\label{L2.11}
  There exists an extremal polynomial $P^*$
  for $\NN\left(\PP_{n,m},
  L_{p,\left(1-\vert x\vert^2\right)^{\la-1/2}}(\BB^m)\right),
  \,p\in[1,\iy),\,\la>0,$
  whose uniform norm is attained on
 the sphere $S^{m-1}$, that is, there exists
 $x_0\in S^{m-1}$ such that
 $\|P^*\|_{L_\iy(\BB^m)}=\vert P^*(x_0)\vert$.
\end{lemma}
\end{example}
\begin{remark}\label{R2.12}
Xu \cite{X2004, X2005a} introduced a weighted version
of the spherical GTO $T_t(F,S^m)$ defined by \eqref{E2.15n}.
Note that  GTO \eqref{E2.18n},
similarly to the Chebyshev-type GTO \eqref{E2.11n},
is the "projection"
of the weighted spherical GTO on $\BB^m$.
In addition, note that, in a sense, GTO \eqref{E2.11n} is
a limit version of GTO \eqref{E2.18n} as $\la\to 0+$
(cf. \cite[Sect. 3.3]{X2005}).
\end{remark}

\section{The $m$-dimensional Cube $Q^m$.
Proof of Theorem  \ref{T1.2a}}\label{S3n}
 \noindent
\setcounter{equation}{0}
Throughout the section we assume that
$V\subset\R^m$ satisfies the $\Pi$-condition and
$\al_j\ge 0,\,\be_j\ge -1/2,\,1\le j\le m$.
The proof of Theorem \ref{T1.2a} is based
on four lemmas below.
We start with three standard properties of
multivariate  entire functions
 of exponential type.

 \begin{lemma}\label{L3.2}
 (a) If $f\in B_V$, then there exists $M=M(V)>0$ such
 that $f\in B_{Q^m(M)}$.\\
 (b) The following crude Nikolskii-type inequalities hold true:
 \bna
   &&\left\|f\right\|_{L_{\iy}(\R^m)}
  \le C_4
  \left\|f\right\|_{L_{p}(\R^m)},\quad f\in B_V\cap L_p(\R^m),\quad
 p\in(0,\iy),\label{E3.4n}\\
 &&\left\|f\right\|_{L_{\iy}(\R^m)}
  \le C_5
  \left\|f\right\|_{L_{p,
  \prod_{j=1}^m\left\vert t_j\right\vert^{\al_j}}(\R^m)},
  \quad f\in B_V\cap L_{p,\prod_{j=1}^m\left\vert t_j
  \right\vert^{\al_j}}(\R^m),
  \quad
 p\in(0,\iy),\label{E3.5n}
  \ena
  where $C_4$ and $C_5$ are independent of $f$.\\
  (c) For any sequence $\{f_n\}_{n=1}^\iy,\,
f_n\in B_V\cap L_\iy(\R^m),\,n\in\N,$
with $\sup_{n\in\N}\| f_n\|_{L_\iy(\R^m)}= C$, there exist a subsequence
$\{f_{n_d}\}_{d=1}^\iy$ and a function $f_0\in B_V\cap L_\iy(\R^m)$
such that
$
\lim_{d\to\iy} f_{n_d}=f_0
$
uniformly on any compact set in $R^m$.
  \end{lemma}
  \begin{proof}
  Statement (a) follows from the obvious inclusion $V\subseteq Q^m(M)$
  for a certain $M=M(V)>0$
  (cf. \cite[Lemma 2.1 (a)]{G2019b}).
  Next, statement (c) was proved in \cite[Lemma 2.3]{G2018},
  and inequality \eqref{E3.4n} was established in \cite[Theorem 5.7]{NW1978}.

  To prove weighted inequality \eqref{E3.5n}, we need the following definition.
  Given $L>0$ and $\de\in(0,L^m]$, we say that a measurable set $E\subseteq\R^m$
  is $(L,\de)$-\textit{dense} if $\vert Q_{y,L}\cap E\vert_m\ge \de$ for any cube
  $Q_{y,L}:=\{t\in\R^m: \max_{1\le j\le m}\vert t_j-y_j\vert\le L/2\},\,y\in\R^m$.
  Kacnelson \cite[Theorem 3]{K1973} proved the following theorem:
  if a set $E\subseteq\R^m$ is $(L,\de)$-dense and
  $f\in B_{Q^m(M)}\cap L_p(E),\,p\in(0,\iy]$, then $f\in L_p(\R^m)$ and
  \beq\label{E3.7n}
  \|f\|_{L_p(\R^m)}\le C_6(M,L,\de)\|f\|_{L_p(E)}.
  \eeq
  If the set
  $E:=\{t\in\R^m:\prod_{j=1}^m\left\vert t_j\right\vert^{\al_j}\ge 1\}$
  is $(L,\de)$-dense for certain $L$ and $\de$,
  then \eqref{E3.5n} follows from \eqref{E3.7n}.
  Indeed, by statement (a), $f\in B_{Q^m(M)}$ for  certain $M>0$. Next
  using \eqref{E3.4n} and Kacnelson's theorem, we obtain
  \ba
  \left\|f\right\|_{L_{\iy}(\R^m)}
  \le C_4
  \left\|f\right\|_{L_{p}(\R^m)}
  \le C_4C_6
  \left\|f\right\|_{L_{p}(E)}
  <C_4C_6  \left\|f\right\|_{L_{p,
  \prod_{j=1}^m\left\vert t_j\right\vert^{\al_j}}(\R^m)}.
  \ea
  Thus \eqref{E3.5n} is established for $C_5=C_4C_6$.
  It remains to show that $E$ is $(L,\de)$-dense
  for certain $L$ and $\de$.
  It is clear by contradiction that
  the following relation holds true:
  \beq\label{E3.8n}
  E^c:=\left\{t\in\R^m:
  \prod_{j=1}^m\left\vert t_j\right\vert^{\al_j}< 1\right\}
  \subseteq \bigcup_{j=1}^m E_j,
  \eeq
  where
  $E_j:=\{t\in\R^m:\vert t_j\vert\le 1\},\,1\le j\le m$.
  It follows from \eqref{E3.8n} that for any cube $Q_{y,L}$ with
  $L>2m$ and $y\in\R^m$,
  \ba
  \left\vert Q_{y,L}\cap E^c\right\vert_m
  \le \sum_{j=1}^m \left\vert Q_{y,L}\cap E_j\right\vert_m
  \le 2mL^{m-1},\quad
  \left\vert Q_{y,L}\cap E\right\vert_m \ge (L-2m)L^{m-1}.
  \ea
  Thus $E$ is $\left(2m+1,(2m+1)^{m-1}\right)$-dense,
  and the lemma is established.
 \end{proof}
  In the next lemma we discuss the error of
  polynomial approximation for functions from $B_V$.

  \begin{lemma}\label{L3.3}
  For any $F\in  B_{V}\cap L_\iy(\R^m),\,\tau\in(0,1)$, and $a\ge 1$,
  there is a polynomial
  $P_{a}\in\PP_{aV}$ such that for
    $r\in(0,\iy]$,
   \beq\label{E3.9n}
  \lim_{a\to\iy}
  \left\|F-P_{a}\right\|_{L_{r,
  \prod_{j=1}^m\left\vert t_j\right\vert^{\al_j}
  \left(1-(t_j/(a\tau))^2\right)^{\be_j}}
  (Q^m(a\tau))}=0.
  \eeq
  \end{lemma}
  \begin{proof}
  A nonweighted version of Lemma \ref{L3.3} was proved
   in \cite[Lemma 2.7]{G2020a}. In particular, the following
   inequality holds true for
   any $F\in  B_{V}\cap L_\iy(\R^m),\,\tau\in(0,1)$, and
    a certain $P_a\in\PP_{aV},\,a\ge 1$
   (see \cite[Eq. (2.28)]{G2020a}):
   \beq\label{E3.10n}
 \left\|F-P_a\right\|_{L_{\iy}(Q^m(a\tau))}
 \le C_{7}(\tau,m,V)\,a^{\frac{2m}{m+4}}
\exp[-C_{8}(\tau,m,V)\, a]
 \left\|F\right\|_{L_{\iy}(\R^m)}.
 \eeq
Using \eqref{E3.10n}, we obtain
 \bna\label{E3.12n}
 &&\left\|F-P_{a}\right\|_{L_{r,
  \prod_{j=1}^m\left\vert t_j\right\vert^{\al_j}
  \left(1-(t_j/(a\tau))^2\right)^{\be_j}}
  (Q^m(a\tau))}\nonumber\\
  &&\le C_{7}
  \left(\prod_{j=1}^m B\left(\al_j/2+1/2,\be_j+1\right)\right)^{1/r}
  a^{\frac{2m}{m+4}+
  \left(m+\sum_{j=1}^m\al_j\right)/r}
\exp[-C_{8}\, a]
 \left\|F\right\|_{L_{\iy}(\R^m)}.
 \ena
 Then \eqref{E3.9n} follows from \eqref{E3.12n}.
  \end{proof}
  A certain inequality
  of different weighted metrics for multivariate polynomials
      is discussed in the following lemma.

    \begin{lemma}\label{L3.4}
     Given  $a\ge 1,\,M>0,\,p\in(0,\iy),\,\tau\in(2/3,1),$
     and $P\in\PP_{aV}$, the
    following inequality holds true:
    \beq\label{E3.13n}
    \|P\|_{L_\iy(Q^m(\tau M))}
    \le C_{9}\,(a/M)^{\left(m+\sum_{j=1}^m\al_j\right)/p}
    \|P\|_{L_{p,
  \prod_{j=1}^m\left\vert t_j\right\vert^{\al_j}
  \left(1-\left(t_j/M\right)^2\right)^{\be_j}}
  (Q^m(M))},
    \eeq
    where $C_9$ is independent of $a,\,M$, and $P$.
    \end{lemma}
    \begin{proof}
    Inequality \eqref{E3.13n} for $\al_j=\be_j=0,\,
    1\le j\le m$, was proved in \cite[Lemma 2.9]{G2020a}.
    Then setting $\al_j=\be_j=0,\,
    1\le j\le m$, and
    replacing $M$ by $(1-\vep/2)M$ and
    $\tau$ by $(1-\vep)/(1-\vep/2),\,\vep\in(0,1/2),$ in \eqref{E3.13n},
    we obtain
     \beq\label{E3.14n}
  \|P\|_{L_\iy(Q^m((1-\vep) M))}
    \le C_{10}\,(a/M)^{m/p}
    \|P\|_{L_{p}(Q^m((1-\vep/2)M))}.
  \eeq
  Therefore,  \eqref{E3.13n} follows from \eqref{E3.14n}
  and from the following inequalities:
  \bna\label{E3.15n}
  &&\|P\|_{L_{p}(Q^m((1-\vep/2)M))}
  \le C_{11}\,(a/M)^{\left(\sum_{j=1}^m\al_j\right)/p}
    \|P\|_{L_{p,
  \prod_{j=1}^m\left\vert t_j\right\vert^{\al_j}}(Q^m((1-\vep/2)M))}
  \nonumber\\
  &&\le C_{12}\,(a/M)^{\left(\sum_{j=1}^m\al_j\right)/p}
    \|P\|_{L_{p,
  \prod_{j=1}^m\left\vert t_j\right\vert^{\al_j}
  \left(1-\left(t_j/M\right)^2\right)^{\be_j}}(Q^m((1-\vep/2)M))},
  \ena
  where $C_{10},\,C_{11}$, and $C_{12}$ in \eqref{E3.14n}
   and \eqref{E3.15n}
   are independent of $a,\,M$, and $P$.
   The second inequality in \eqref{E3.15n} is all but trivial.
  To prove the first inequality in \eqref{E3.15n}, we observe that
  \bna\label{E3.16n}
   \|P\|_{L_{p,
  \prod_{j=1}^m\left\vert t_j\right\vert^{\al_j}}(Q^m((1-\vep/2)M))}
  &\ge& \|P\|_{L_{p,
  \prod_{j=1}^m\left\vert t_j\right\vert^{\al_j}}(Q^m((1-\vep/2)M)
  \setminus Q^m(CM/a)) }\nonumber\\
  &\ge& (CM/a)^{\left(\sum_{j=1}^m\al_j\right)/p}
  \|P\|_{L_{p}
  (Q^m((1-\vep/2)M)\setminus Q^m(CM/a)) },
  \ena
  where $C\in(0,1/3)$ is a fixed number.
  Next, we note that
  $0<C/a<1/3<1-\vep$, so
  by \eqref{E3.14n},
  \beq\label{E3.17n}
  \|P\|_{L_{p}(Q^m(CM/a))}
  \le \left(2CM/a\right)^{m/p}\|P\|_{L_{\iy}((Q^m(CM/a)))}
  \le \left(2C\right)^{m/p}C_{10}
  \|P\|_{L_{p}(Q^m((1-\vep/2)M))}.
  \eeq
  Choosing now $C:=\min\{1/3,C_{10}^{-p/m}2^{-1/m-1}\}$,
  we obtain from \eqref{E3.17n}
  \bna\label{E3.18n}
 && \|P\|_{L_{p}
  (Q^m((1-\vep/2)M)\setminus Q^m(CM/a))}^p
  =\|P\|_{L_{p}(Q^m((1-\vep/2)M))}^p
  -\|P\|_{L_{p}( Q^m(CM/a))}^p\nonumber\\
  &&\ge \left(1-(2C)^mC_{10}^p\right)\|P\|_{L_{p}(Q^m((1-\vep/2)M))}^p
  \ge (1/2)\|P\|_{L_{p}(Q^m((1-\vep/2)M))}^p.
  \ena
  Finally, combining \eqref{E3.16n} and \eqref{E3.18n}, we arrive at
  the first inequality in \eqref{E3.15n}. Thus \eqref{E3.13n} is established.
    \end{proof}
    We also need the following technical lemma.

     \begin{lemma}\label{L3.5}
     The following estimates hold true
($\tau\in(0,1),\, v_j\in[-\pi/2,\pi/2],\,1\le j\le m$):
\beq \label{E3.18an}
0\le \prod_{j=1}^m\left\vert  v_j\right\vert^{\al_j}
-\prod_{j=1}^m\left\vert\sin v_j\right\vert^{\al_j}
\left(\tau^2\cos^2 v_j+1-\tau^2\right)^{\be_j}\cos v_j
\le C_{13}\prod_{j=1}^m\left\vert  v_j\right\vert^{\al_j}\sum_{j=1}^mv_j^{2},
\eeq
where $C_{13}:=\max\{1,\al_1,\ldots,\al_m,2\be_1+1,\ldots,2\be_m+1\}$.
\end{lemma}
\begin{proof}
The proof is based on the following elementary inequalities
\ba
&&1-\theta^\g\le \max\{1,\g\}(1-\theta),
\quad \theta\in[0,1],\quad\g\ge 0;
\qquad 1-\sin \theta/\theta\le \theta^2/6,\quad\theta\in[0,\pi/2];\\
&&(\cos\theta)^{2\g_1+1}
\le\left(\tau^2\cos^2 \theta+1-\tau^2\right)^{\g}\cos \theta
\le (\cos\theta)^{2\g_2+1}\le 1,\, \theta\in [0,\pi/2],\, \tau\in(0,1),
\, \g\ge -1/2;
\ea
where $\g_1:=\max\{\g,0\}$ and $\g_2:=\min\{\g,0\}$.
In particular, the left inequality in \eqref{E3.18an}
immediately follows from the elementary inequality
$\left(\tau^2\cos^2 \theta+1-\tau^2\right)^{\g}\cos \theta
\le 1,\,\theta\in [0,\pi/2]$, given above.

Next, for $v_j\in[0,\pi/2],\,1\le j\le m$, we obtain
\bna \label{E3.18bn}
&&\prod_{j=1}^mv_j^{\al_j}
-\prod_{j=1}^m
\left(\sin v_j\right)^{\al_j}
\left(\tau^2\cos^2 v_j+1-\tau^2\right)^{\be_j}\cos v_j\nonumber\\
&&\le \prod_{j=1}^mv_j^{\al_j}\left(
1-\prod_{j=1}^m\left(\frac{\sin v_j}{v_j}\right)^{\al_j}
+1-\prod_{j=1}^m
\left(\tau^2\cos^2 v_j+1-\tau^2\right)^{\be_j}\cos v_j\right).
\ena
Using now the easy identity
\ba
1-\prod_{j=1}^m D_j
=\sum_{j=1}^m\left(1-D_j\right)
\prod_{l=j+1}^m D_l,\qquad
D_j\in \R^1,\quad 1\le j\le m,
\ea
where
$\prod_{l=j+1}^m:=1$ for $m<j+1$,
and, in addition, using the elementary inequalities given above,
we obtain
\bna
&&1-\prod_{j=1}^m\left(\frac{\sin v_j}{v_j}\right)^{\al_j}
=\sum_{j=1}^m\left(1-\left(\frac{\sin v_j}
{v_j}\right)^{\al_j}\right)
\prod_{l=j+1}^m\left(\frac{\sin v_l}
{v_l}\right)^{\al_l}\nonumber\\
&&\le  \sum_{j=1}^m\left(1-\left(\frac{\sin v_j}
{v_j}\right)^{\al_j}\right)
\le(1/6)\max\{1,\al_1,\ldots,\al_m\} \sum_{j=1}^mv_j^{2};
\label{E3.18cn}\\
&&1-\prod_{j=1}^m
\left(\tau^2\cos^2 v_j+1-\tau^2\right)^{\be_j}\cos v_j
\nonumber\\
&&= \sum_{j=1}^m\left(1-
\left(\tau^2\cos^2 v_j+1-\tau^2\right)^{\be_j}\cos v_j\right)
\prod_{l=j+1}^m
\left(\tau^2\cos^2 v_l+1-\tau^2\right)^{\be_l}\cos v_l
\nonumber\\
&&\le \sum_{j=1}^m\left(1-
\left(\tau^2\cos^2 v_j+1-\tau^2\right)^{\be_j}\cos v_j\right)
\le \sum_{j=1}^m\left(1-
\left(\cos v_j\right)^{2\max\{\be_j,0\}+1}\right)\nonumber\\
&&\le (1/2)\max\{1,2\be_1+1,\ldots,2\be_m+1\} \sum_{j=1}^mv_j^{2}
\label{E3.18dn}.
\ena
Thus the right inequality in \eqref{E3.18an}
follows from \eqref{E3.18bn}, \eqref{E3.18cn},
 and \eqref{E3.18dn}.
 \end{proof}
\noindent
\emph{Proof of Theorem \ref{T1.2a}.}
Throughout the proof we use the notation $\tilde{p}=\min\{1,p\},\,
p\in(0,\iy)$, introduced in Section \ref{S1}.\vspace{.12in}\\
\textbf{Step 1.}
We first prove the inequality
\bna \label{E3.21n}
&&\NN_0\left(B_{V}\cap L_{p,\prod_{j=1}^m\left\vert t_j
\right\vert^{\al_j}}(\R^m),
  L_{p,\prod_{j=1}^m\left\vert t_j
  \right\vert^{\al_j}}(\R^m)\right)\nonumber\\
 && \le \liminf_{a\to\iy}a^{-\left(m+\sum_{j=1}^m\al_j\right)/p}
 \NN_{0}\left(\PP_{aV},
  L_{p,\prod_{j=1}^m\left\vert t_j\right\vert^{\al_j}
  \left(1-t_j^2\right)^{\be_j}}(Q^m)\right).
  \ena

Let $f$ be any function from $B_{V}\cap
L_{p,\prod_{j=1}^m\vert t_j\vert^{\al_j}}(\R^m),
  \,p\in(0,\iy)$,
  and let $\tau\in(0,1)$ be a fixed number.
Then $f\in L_\iy(\R^m)$ by Nikolskii-type inequality \eqref{E3.5n}.

In addition, we need one more function related to $f$.
Let $\g=\g(V)>0$ be a smallest number such that
$Q^m\subseteq \g V$.
Given $\vep\in(0,1/(2\g))$, we define
an $L_{p,\prod_{j=1}^m\left\vert t_j\right\vert^{\al_j}}(\R^m)$-
version of $f$ by
 \beq\label{E3.22n}
 F(t)=F_{\vep}(t)
 :=f((1-\g\vep)t)\left(\prod_{j=1}^m\frac{\sin
 \left(\vep t_j/d\right)}{\vep t_j/d}\right)^{d},
 \eeq
 where $d:=\lfloor(m+\sum_{j=1}^m\al_j)/p\rfloor +1$.
 Then $F$ has exponential type
 $(1-\g\vep)V+\vep Q^m \subseteq (1-\g\vep)V+ \g\vep V=V$.
 Therefore,
 $F\in B_{V}\cap L_\iy(\R^m)$ and, in addition,
 \beq\label{E3.23n}
 \|F\|_{L_{p,\prod_{j=1}^m\left\vert t_j\right\vert^{\al_j}}(\R^m)}
 \le (1-\g\vep)^{-\left(m+\sum_{j=1}^m\al_j\right)/p}\|f\|_
 {L_{p,\prod_{j=1}^m\left\vert t_j\right\vert^{\al_j}}(\R^m)}.
 \eeq

Next, we use polynomials  $P_{a}\in\PP_{aV},\,a\ge 1$,
from Lemma \ref{L3.3} such that for $r=\iy$ or $r=p, p\in(0,\iy)$,
 the following relations hold true
by \eqref{E3.9n}:
\beq\label{E3.24n}
\lim_{a\to\iy}
  \left\|F-P_{a}\right\|_{L_{\iy}(Q^m(a\tau))}=0,\,\,\,
  \lim_{a\to\iy}
  \left\|F-P_{a}\right\|_{L_{p,
  \prod_{j=1}^m\left\vert t_j\right\vert^{\al_j}
  \left(1-(t_j/(a\tau))^2\right)^{\be_j}}
  (Q^m(a\tau))}=0.
  \eeq
Using the first limit equality of \eqref{E3.24n}
and the definition of the sharp constant given by \eqref{E1.2},
we obtain
\bna\label{E3.25n}
 \left\vert f(0)\right\vert
 &=&\left\vert F(0)\right\vert
 \le  \lim_{a\to\iy}\left\vert F(0)-P_{a}(0)\right\vert
 +\liminf_{a\to\iy}\left\vert P_{a}(0)\right\vert
 =\liminf_{a\to\iy}\left\vert P_{a}(0)\right\vert\nonumber\\
 &&\le \liminf_{a\to\iy}
 \NN_{0}\left(\PP_{aV},
  L_{p,\prod_{j=1}^m\left\vert t_j\right\vert^{\al_j}
  \left(1-\left(t_j/(a\tau)\right)^2\right)^{\be_j}}
  (Q^m(a\tau))\right)\nonumber\\
  &&\times\limsup_{a\to\iy}\left\| P_{a}\right\|_
  { L_{p,\prod_{j=1}^m\left\vert t_j\right\vert^{\al_j}
  \left(1-\left(t_j/(a\tau)\right)^2\right)^{\be_j}}
  (Q^m(a\tau))}\nonumber\\
 && =\liminf_{a\to\iy}(a\tau)^{-\left(m+\sum_{j=1}^m\al_j\right)/p}
  \NN_{0}\left(\PP_{aV},
  L_{p,\prod_{j=1}^m\left\vert t_j\right\vert^{\al_j}
  \left(1-t_j^2\right)^{\be_j}}(Q^m)\right)\nonumber\\
  &&\times\limsup_{a\to\iy}\left\| P_{a}\right\|_
  { L_{p,\prod_{j=1}^m\left\vert t_j\right\vert^{\al_j}
  \left(1-\left(t_j/(a\tau)\right)^2\right)^{\be_j}}
  (Q^m(a\tau))}.
 \ena
 It remains to estimate the last line in \eqref{E3.25n}.
 Using "triangle" inequality \eqref{E1.1} and the second
 relations of \eqref{E3.24n}, we have
 \bna\label{E3.26n}
&&\limsup_{a\to\iy}\left\| P_{a}\right\|^{\tilde{p}}_
  { L_{p,\prod_{j=1}^m\left\vert t_j\right\vert^{\al_j}
  \left(1-\left(t_j/(a\tau)\right)^2\right)^{\be_j}}
  (Q^m(a\tau))}\nonumber\\
 && \le \lim_{a\to\iy}\left\|F- P_{a}\right\|^{\tilde{p}}_
  { L_{p,\prod_{j=1}^m\left\vert t_j\right\vert^{\al_j}
  \left(1-\left(t_j/(a\tau)\right)^2\right)^{\be_j}}
  (Q^m(a\tau))}\nonumber\\
 && +\limsup_{a\to\iy}\left\| F\right\|^{\tilde{p}}_
  { L_{p,\prod_{j=1}^m\left\vert t_j\right\vert^{\al_j}
  \left(1-\left(t_j/(a\tau)\right)^2\right)^{\be_j}}
  (Q^m(a\tau))}\nonumber\\
 &&=
 \limsup_{a\to\iy}I^{\tilde{p}}(a)
 :=\limsup_{a\to\iy}\left\| F\right\|^{\tilde{p}}_
  { L_{p,\prod_{j=1}^m\left\vert t_j\right\vert^{\al_j}
  \left(1-\left(t_j/(a\tau)\right)^2\right)^{\be_j}}
  (Q^m(a\tau))}.
  \ena
  Furthermore, we prove the estimate
  \beq\label{E3.27n}
  \limsup_{a\to\iy} I(a)\le (1-\g\vep)^{-\left(m+\sum_{j=1}^m
  \al_j\right)/p}\|f\|_
 {L_{p,\prod_{j=1}^m\left\vert t_j\right\vert^{\al_j}}(\R^m)}.
  \eeq
  Indeed, setting
$\Omega_\nu(A):=\left\{t\in\R^m:
\left\vert t_\nu\right\vert\ge A
\right\},\,A>0,$
and using the estimate
\ba
\left\vert F(t)\right\vert\le \|f\|_{L_\iy(\R^m)}
\left(\vep
\left\vert t_\nu\right\vert/d\right)^{-d},
\qquad t\in \Omega_\nu(A),
\ea
that follows from \eqref{E3.22n},
we obtain
  for any $\de\in(0,1)$
\bna\label{E3.28n}
&&\limsup_{a\to\iy}
\left\|F\right\|_
  { L_{p,\prod_{j=1}^m\left\vert t_j\right\vert^{\al_j}
  \left(1-\left(t_j/(a\tau)\right)^2\right)^{\be_j}}
  \left(Q^m(a\tau)\cap \Omega_\nu(\de a\tau)\right)}\nonumber\\
&&\le C_{14}\|f\|_{L_\iy(\R^m)}
\limsup_{a\to\iy} a^{\left(m+\sum_{j=1}^m\al_j
\right)/p-d}=0,\qquad 1\le \nu\le m,
\ena
where
\ba
C_{14}\le \left(\prod_{j=1}^mB\left(\al_j/2+1/2,\be_j+1
\right)\right)^{1/p}
\tau^{\left(m+\sum_{j=1}^m\al_j\right)/p}
(\vep\de \tau/d)^{-d}.
\ea
Next, without loss of generality we assume that
there exists $l\in\Z^1_+,\,0\le l\le m,$ such that
  $-1/2\le\be_j<0,\,1\le j\le l$, and
$\be_j\ge 0,\,l+1\le j\le m$.
Then it follows from \eqref{E3.28n}
 and \eqref{E3.23n} that
\bna\label{E3.29n}
&&\limsup_{a\to\iy} I^p(a)
\le \limsup_{a\to\iy}\left\| F\right\|^{p}_
  { L_{p,\prod_{j=1}^m\left\vert t_j\right\vert^{\al_j}
  \left(1-\left(t_j/(a\tau)\right)^2\right)^
  {\be_j}}(Q^m(\de a\tau))}
  \nonumber\\
  &&+\limsup_{a\to\iy}
  \sum_{\nu=1}^m\left\| F\right\|^{p}_
  {L_{p,\prod_{j=1}^m\left\vert t_j\right\vert^{\al_j}
  \left(1-\left(t_j/(a\tau)\right)^2\right)^
  {\be_j}}\left(Q^m(a\tau)\cap \Omega_\nu(\de a\tau)\right)}
 \nonumber\\
 &&\le \limsup_{a\to\iy}\left\| F\right\|^{p}_
  { L_{p,\prod_{j=1}^m\left\vert t_j\right\vert^{\al_j}
  \prod_{j=1}^l\left(1-\left(t_j/(a\tau)\right)^2\right)^
  {\be_j}}(Q^m(\de a\tau))}
  \nonumber\\
  &&\le (1-\de^2)^{\sum_{j=1}^l\be_j}
  (1-\g\vep)^{-\left(m+\sum_{j=1}^m\al_j\right)}
  \|f\|^p_
 {L_{p,\prod_{j=1}^m\left\vert t_j\right\vert^{\al_j}}(\R^m)}.
\ena
Letting $\de\to 0+$ in the last line of \eqref{E3.29n},
we arrive at \eqref{E3.27n} from \eqref{E3.29n}.
Combining \eqref{E3.25n}, \eqref{E3.26n}, and \eqref{E3.27n},
we obtain
\bna\label{E3.30n}
&&\frac{\left\vert f(0)\right\vert}
{\|f\|^p_
 {L_{p,\prod_{j=1}^m\vert t_j\vert^{\al_j}}(\R^m)}}
  \le [(1-\g\vep)\tau]
^{-\left(m+\sum_{j=1}^m\al_j\right)/p}\nonumber\\
&&\times \liminf_{a\to\iy}a^{-\left(m+\sum_{j=1}^m\al_j\right)/p}
  \NN_{0}\left(\PP_{aV},
  L_{p,\prod_{j=1}^m\left\vert t_j\right\vert^{\al_j}
  \left(1-t_j^2\right)^{\be_j}}(Q^m)\right).
  \ena
  Finally, letting $\vep\to 0+$ and
  $\tau\to 1-$  in \eqref{E3.30n},
we complete the proof of \eqref{E3.21n}.\vspace{.12in}\\
\textbf{Step 2.}
Furthermore, we will prove the
 inequality
 \bna\label{E3.31n}
 && \limsup_{a\to\iy}a^{-\left(m+\sum_{j=1}^m\al_j\right)/p}
 \NN_{0}\left(\PP_{aV},
  L_{p,\prod_{j=1}^m\left\vert t_j\right\vert^{\al_j}
  \left(1-t_j^2\right)^{\be_j}}(Q^m)\right)\nonumber\\
  &&\le\NN_0\left(B_{V}\cap L_{p,\prod_{j=1}^m\left\vert t_j
  \right\vert^{\al_j}}(\R^m),
  L_{p,\prod_{j=1}^m\left\vert t_j\right\vert^{\al_j}}(\R^m)\right),
  \ena
  by constructing a nontrivial function $f_0\in B_V\cap
  L_{p,\prod_{j=1}^m\left\vert t_j\right\vert^{\al_j}}(\R^m)$
   such that
 \bna \label{E3.32n}
  && \limsup_{a\to\iy}a^{-\left(m+\sum_{j=1}^m\al_j\right)/p}
 \NN_{0}\left(\PP_{aV},
  L_{p,\prod_{j=1}^m\left\vert t_j\right\vert^{\al_j}
  \left(1-t_j^2\right)^{\be_j}}(Q^m)\right)\nonumber\\
  &&\le\left\vert f_0(0)\right\vert
  /\left\|f_0\right\|_{L_{p,\prod_{j=1}^m\left\vert t_j
  \right\vert^{\al_j}}(\R^m)}
  \nonumber\\
  &&\le\NN_0\left(B_{V}\cap L_{p,\prod_{j=1}^m
  \vert t_j\vert^{\al_j}}(\R^m),
  L_{p,\prod_{j=1}^m\vert t_j\vert^{\al_j}}(\R^m)\right).
 \ena
 Then inequalities \eqref{E3.21n} and \eqref{E3.31n} imply  \eqref{E1.8a}.
 In addition, $f_0$ is an extremal function for
 $\NN_0\left(B_{V}\cap
  L_{p,\prod_{j=1}^m\vert t_j\vert^{\al_j}}(\R^m),
  L_{p,\prod_{j=1}^m\vert t_j\vert^{\al_j}}(\R^m)\right)$,
that is, \eqref{E1.9a} is valid.

It remains to construct a nontrivial function $f_0$,
satisfying \eqref{E3.32n}.
We first note that
\beq \label{E3.33n}
\inf_{a\ge 1}
a^{-\left(m+\sum_{j=1}^m\al_j\right)/p}
 \NN_{0}\left(\PP_{aV},
  L_{p,\prod_{j=1}^m\left\vert t_j\right\vert^{\al_j}
  \left(1-t_j^2\right)^{\be_j}}(Q^m)\right)
\ge C_{15},
\eeq
where $C_{15}$ is independent of $a$.
This inequality follows immediately from  \eqref{E3.21n}.

Let $U_a\in\PP_{aV}$ be a polynomial, satisfying the equality
\bna \label{E3.34n}
 \NN_{0}\left(\PP_{aV},
  L_{p,\prod_{j=1}^m\left\vert t_j\right\vert^{\al_j}
  \left(1-t_j^2\right)^{\be_j}}(Q^m)\right)
  =\frac{\left\vert U_a(0)\right\vert}
{\|U_a\|_{L_{p,\prod_{j=1}^m\left\vert t_j\right\vert^{\al_j}
  \left(1-t_j^2\right)^{\be_j}}(Q^m)}},\qquad a\ge 1.
\ena
The existence of an extremal polynomial $U_a$ in \eqref{E3.34n}
can be proved by the standard compactness argument
(see, e.g., \cite[Proof of Theorem 1.5]{GT2017} and
\cite[Proof of Theorem 1.3]{G2018}).
Next, setting $P_a(t):=U_a(t/a)$,
 we have from \eqref{E3.34n} that
\bna \label{E3.35n}
&&a^{-\left(m+\sum_{j=1}^m\al_j\right)/p}
 \NN_{0}\left(\PP_{aV},
  L_{p,\prod_{j=1}^m\left\vert t_j\right\vert^{\al_j}
  \left(1-t_j^2\right)^{\be_j}}(Q^m)\right)\nonumber\\
&&=\frac{\left\vert P_a(0)\right\vert}
{\|P_a\|_{L_{p,\prod_{j=1}^m\left\vert t_j\right\vert^{\al_j}
  \left(1-\left(t_j/a\right)^2\right)^{\be_j}}(Q^m(a))}}\nonumber\\
  &&=\frac{1}{\|P_a\|_{L_{p,\prod_{j=1}^m
  \left\vert t_j\right\vert^{\al_j}
  \left(1-\left(t_j/a\right)^2\right)^{\be_j}}(Q^m(a))}},
\ena
since we can assume that
\beq \label{E3.36n}
\left\vert P_a(0)\right\vert=1.
\eeq
Then it follows from \eqref{E3.35n} and \eqref{E3.33n}
that
\ba
\sup_{a\ge 1}
\|P_a\|_{L_{p,\prod_{j=1}^m\left\vert t_j\right\vert^{\al_j}
  \left(1-\left(t_j/a\right)^2\right)^{\be_j}}(Q^m(a))}
\le 1/C_{15}.
\ea
Hence using Lemma \ref{L3.4} for $M=a$
 and $\tau\in(2/3,1)$,
we obtain  the estimate
\beq \label{E3.37n}
\sup_{a\ge 1}\|P_a\|_{L_\iy(Q^m(a\tau))}
\le C_{9}/C_{15}
= C_{16}.
\eeq

Furthermore, given $a\ge 1$ and $\tau\in(2/3,1)$
we define the trigonometric polynomial
\beq\label{E3.37an}
R_{a,\tau}(y)
:=P_a(a\tau\sin(y_1/(a\tau)),\ldots, a\tau\sin(y_m/(a\tau))),\qquad
y\in \R^m,
\eeq
which is a periodic version of $P_a$.
Then $R_{a,\tau}$ satisfies the following properties:
\begin{itemize}
\item[(P1)] $R_{a,\tau}\in B_{(1/\tau)V}$.
\item[(P2)] The following relations hold true:
\beq \label{E3.38n}
\sup_{a\ge 1}\|R_{a,\tau}\|_{L_\iy(Q^m(a\tau\pi/2)}
=\sup_{a\ge 1}\|R_{a,\tau}\|_{L_\iy(\R^m)}\le C_{16}.
\eeq
\item[(P3)]
$ R_{a,\tau}(0)=P_{a}(0).$
\item[(P4)]
 For $a\ge 1,\,\tau\in(2/3,1),\,p\in(0,\iy)$, and
$M\in\left(0,a\tau/\sqrt{m C_{13}}\,\right)$,
\bna \label{E3.39n}
&&\|P_a\|_{L_{p,\prod_{j=1}^m\left\vert t_j\right\vert^{\al_j}
  \left(1-\left(t_j/a\right)^2\right)^{\be_j}}(Q^m(a))}
  \nonumber\\
&&\ge \left(1-C_{13}mM^2(a\tau)^{-2}\right)^{1/p}
\|R_{a,\tau}\|_{L_{p,\prod_{j=1}^m\left\vert y_j\right
\vert^{\al_j}}(Q^m(M))},
\ena
where the constant
$C_{13}\ge 1$ is defined in Lemma \ref{L3.5}.
\end{itemize}
Indeed, property (P3) is trivial, and property (P2)
is an immediate consequence of \eqref{E3.37n}
and \eqref{E3.37an}.
Next, setting $b:=a\tau$, we see that
\ba
R_{a,\tau}(by)
=\sum_{k\in aV\cap\Z^m_+}
b^{\langle k\rangle}c_k\prod_{j=1}^m\sin^{k_j}y_j
=\sum_{k\in aV\cap\Z^m_+}
b^{\langle k\rangle}c_k
\sum_{\eta\in\Z^m,0\le\vert \eta_j
\vert\le k_j,1\le j\le m}
d_{\eta,k}\exp[i(\eta, y)].
\ea
 Then $R_{a,\tau}(b\cdot)\in\TT_{aV}$,
 since $V$ satisfies the $\Pi$-condition,
and therefore, $R_{a,\tau}(\cdot)\in B_{(a/b)V}$.
This proves property (P1).
To prove (P4), we first discuss estimates for
the weight
\ba
W_{\al,\be}(y)
:=(a\tau)^{\sum_{j=1}^m\al_j}\prod_{j=1}^m
\left\vert\sin(y_j/(a\tau))\right \vert^{\al_j}
\left(\tau^2\cos^2(y_j/(a\tau))+1-\tau^2
\right)^{\be_j}\cos\left(y_j/(a\tau)\right).
\ea
on $Q^m(a\tau\pi/2)$.
Using Lemma \ref{L3.5} for $v_j=y_j/(a\tau),\,
1\le j\le m$, and
$y\in Q^m(M),\,M\in (0,a\tau\pi/2),$ we obtain
\beq\label{E3.41n}
0\le  \prod_{j=1}^m\left\vert y_j\right
\vert^{\al_j}-W_{\al,\be}(y)
\le C_{13}mM^2(a\tau)^{-2}
\prod_{j=1}^m\left\vert y_j\right\vert^{\al_j}.
\eeq
Next, it follows  from \eqref{E3.37an} and
the left inequality in \eqref{E3.41n}
that for $a\ge 1,\,\tau\in(2/3,1),\,p\in(0,\iy)$, and
$M\in(0,a\tau\pi/2]$,
\bna\label{E3.40n}
&&\|P_a\|^p_{L_{p,\prod_{j=1}^m\left\vert t_j\right
\vert^{\al_j}
  \left(1-\left(t_j/a
  \right)^2\right)^{\be_j}}(Q^m(a))}
  \ge \|P_a\|^p_{L_{p,\prod_{j=1}^m\left\vert t_j\right
\vert^{\al_j}
  \left(1-\left(t_j/a
  \right)^2\right)^{\be_j}}(Q^m(a\tau))}
  \nonumber\\
&&=\int_{Q^m(a\tau\pi/2)}\left\vert
R_{a,\tau}(y)\right\vert^p W_{\al,\be}(y)\,dy
\ge \int_{Q^m(M)}\left\vert
R_{a,\tau}(y)\right\vert^p W_{\al,\be}(y)\,dy
\nonumber\\
&&\ge \|R_{a,\tau}\|^p_{L_{p,\prod_{j=1}^m\left\vert y_j\right
\vert^{\al_j}}(Q^m(M))}
-\|R_{a,\tau}\|^p_{L_{p,\prod_{j=1}^m\left\vert y_j\right
\vert^{\al_j}-W_{\al,\be}(y)}(Q^m(M))}.
\ena
Thus inequality \eqref{E3.39n} of property (P4)
for $M\in\left(0,a\tau/\sqrt{m C_{13}}\,\right)$ follows
from \eqref{E3.40n} and
the right inequality in \eqref{E3.41n}.

Let $\{a_n\}_{n=1}^\iy$ be an increasing sequence
of numbers such that
 $\inf_{n\in\N}a_n\ge 1,\,\lim_{n\to\iy}a_n=\iy$,
and
\bna \label{E3.42n}
&&\limsup_{a\to\iy}
a^{-\left(m+\sum_{j=1}^m\al_j\right)/p}
 \NN_{0}\left(\PP_{aV},
  L_{p,\prod_{j=1}^m\left\vert t_j\right\vert^{\al_j}
  \left(1-t_j^2\right)^{\be_j}}(Q^m)\right)\nonumber\\
&&=\lim_{n\to\iy}
a_n^{-\left(m+\sum_{j=1}^m\al_j\right)/p}
 \NN_{0}\left(\PP_{a_nV},
  L_{p,\prod_{j=1}^m\left\vert t_j\right\vert^{\al_j}
  \left(1-t_j^2\right)^{\be_j}}(Q^m)\right).
\ena

Property (P1) and relation \eqref{E3.38n}
of property (P2)
 show that the  sequence of trigonometric polynomials
$\{R_{a_n,\tau}\}_{n=1}^\iy
=\{f_n\}_{n=1}^\iy$
 satisfies the conditions
of Lemma \ref{L3.2} (c)
with $B_V$ replaced by $B_{(1/\tau)V}$.
Therefore, there exist
a subsequence
 $\{R_{a_{n_d},\tau}\}_{d=1}^\iy$ and
 a function $f_{0,\tau}\in B_{(1/\tau)V}$ such that
 \beq \label{E3.43n}
  \lim_{d\to\iy}R_{a_{n_d},\tau}=f_{0,\tau}
 \eeq
 uniformly on any cube $Q^m(M),\,M>0$.

Moreover, it follows from
  \eqref{E3.43n}, property (P3), and \eqref{E3.36n}
  that
  \beq \label{E3.44n}
  \left\vert f_{0,\tau}(0)\right\vert
  =\lim_{d\to\iy}\left\vert R_{a_{n_d},\tau}(0)
  \right\vert
  =\lim_{d\to\iy}\left\vert P_{a_{n_d}}(0)
  \right\vert
  =1.
\eeq
In addition,
using \eqref{E1.1}, \eqref{E3.43n}, \eqref{E3.39n},
\eqref{E3.35n}, and \eqref{E3.42n},
we obtain for any cube $Q^m(M),\,M>0$,
\bna \label{E3.45n}
&&\|f_{0,\tau}\|_{L_{p,\prod_{j=1}^m\left\vert y_j\right
\vert^{\al_j}}(Q^m(M))}\nonumber\\
&&= \lim_{d\to\iy}
\left(\pm\left\|f_{0,\tau}-R_{a_{n_d},\tau}
\right\|_{L_{p,\prod_{j=1}^m\left\vert y_j\right
\vert^{\al_j}}(Q^m(M))}^{\tilde{p}}
+\left\|R_{a_{n_d},\tau}\right\|_
{L_{p,\prod_{j=1}^m\left\vert y_j\right
\vert^{\al_j}}(Q^m(M))}^{\tilde{p}}
\right)^{1/\tilde{p}}\nonumber\\
&&=\lim_{d\to\iy}\left\|R_{a_{n_d},\tau}\right\|_
{L_{p,\prod_{j=1}^m\left\vert y_j\right
\vert^{\al_j}}(Q^m(M))}\nonumber\\
&&\le \lim_{d\to\iy}\left\|P_{a_{n_d}}\right\|_
{L_{p,\prod_{j=1}^m\left\vert t_j\right\vert^{\al_j}
  \left(1-\left(t_j/\left(a_{n_d}\right)
  \right)^2\right)^{\be_j}}
  \left(Q^m\left(a_{n_d}\right)\right)}
\nonumber\\
&&=1/\lim_{d\to\iy}
a_{n_d}^{-\left(m+\sum_{j=1}^m\al_j\right)/p}
 \NN_{0}\left(\PP_{a_{n_d}V},
  L_{p,\prod_{j=1}^m\left\vert t_j\right\vert^{\al_j}
  \left(1-t_j^2\right)^{\be_j}}(Q^m)\right).
\ena
Next using \eqref{E3.33n} and \eqref{E3.45n}, we see that
\beq \label{E3.46n}
\|f_{0,\tau}\|_{L_{p,\prod_{j=1}^m\left\vert y_j\right
\vert^{\al_j}}(Q^m(M))}\le 1/C_{15}.
\eeq
Therefore, $f_{0,\tau}$ is a nontrivial function from
$B_{(1/\tau)V}\cap L_{p,\prod_{j=1}^m\left\vert y_j\right
\vert^{\al_j}}(\R^m)$
by \eqref{E3.46n} and \eqref{E3.44n}.
Thus for any cube $Q^m(M),\,M>0$,
we obtain from \eqref{E3.42n}, \eqref{E3.35n},
 \eqref{E3.39n}, \eqref{E3.43n}, and \eqref{E3.44n}
\bna \label{E3.47n}
&&\limsup_{a\to\iy}
a^{-\left(m+\sum_{j=1}^m\al_j\right)/p}
 \NN_{0}\left(\PP_{aV},
  L_{p,\prod_{j=1}^m\left\vert t_j\right\vert^{\al_j}
  \left(1-t_j^2\right)^{\be_j}}(Q^m)\right)\nonumber\\
&&=\lim_{d\to\iy}\left(\left\|P_{a_{n_d}}\right\|_
{L_{p,\prod_{j=1}^m\left\vert t_j\right\vert^{\al_j}
  \left(1-\left(t_j/\left(a_{n_d}\right)
  \right)^2\right)^{\be_j}}
  \left(Q^m\left(a_{n_d}\right)\right)}\right)^{-1}\nonumber\\
&&\le \lim_{d\to\iy}\left(\left\| R_{a_{n_d},\tau}\right\|
_{L_{p,\prod_{j=1}^m\left\vert y_j\right
\vert^{\al_j}}(Q^m(M))}\right)^{-1}\nonumber\\
&&=\left\vert f_{0,\tau}(0)\right\vert/
\|f_{0,\tau}\|_{L_{p,\prod_{j=1}^m\left\vert y_j\right
\vert^{\al_j}}(Q^m(M))}.
\ena
It follows from \eqref{E3.47n} that
for $\tau\in (2/3,1)$,
\bna\label{E3.48n}
 && \limsup_{a\to\iy}a^{-\left(m+\sum_{j=1}^m\al_j\right)/p}
 \NN_{0}\left(\PP_{aV},
  L_{p,\prod_{j=1}^m\left\vert t_j\right\vert^{\al_j}
  \left(1-t_j^2\right)^{\be_j}}(Q^m)\right)\nonumber\\
  &&\le\NN_0\left(B_{(1/\tau)V}
  \cap L_{p,\prod_{j=1}^m\left\vert t_j\right\vert^{\al_j}}(\R^m),
  L_{p,\prod_{j=1}^m\vert t_j\vert^{\al_j}}(\R^m)\right)\nonumber\\
  &&=\tau^{-\left(m+\sum_{j=1}^m\al_j\right)/p}
  \NN_0\left(B_{V}\cap L_{p,\prod_{j=1}^m\left\vert t_j
  \right\vert^{\al_j}}(\R^m),
  L_{p,\prod_{j=1}^m\left\vert t_j\right\vert^{\al_j}}(\R^m)\right).
  \ena
Then letting $\tau\to 1-$ in
\eqref{E3.48n}, we arrive at \eqref{E3.31n}.
 However, we also need to prove  stronger relations
\eqref{E3.32n}.

To construct $f_0$, recall first that
$f_{0,\tau}(\cdot)\in B_{(1/\tau)V},\,
\tau\in(2/3,1)$,
and by \eqref{E3.46n}
and \eqref{E3.5n},
\ba
&&\sup_{\tau\in (2/3,1)}\|f_{0,\tau}\|_{L_\iy(\R^m)}
=\sup_{\tau\in (2/3,1)}\|f_{0,\tau}(\tau\cdot)\|_{L_\iy(\R^m)}\\
&&\le C_5 \sup_{\tau\in (2/3,1)}
\tau^{-\left(m+\sum_{j=1}^m\al_j\right)/p}
\|f_{0,\tau}\|_{L_{p,\prod_{j=1}^m\left\vert t_j
\right\vert^{\al_j}}(\R^m)}
\le (3/2)^{\left(m+\sum_{j=1}^m\al_j\right)/p}C_5/C_{15}=C,
\ea
where $C_5,\,C_{15}$, and $C$ are independent of $\tau$
because $f_{0,\tau}(\tau\cdot)\in B_V$.
Therefore, by Lemma \ref{L3.2} (c) applied to a sequence
$\left\{f_{0,\tau_n}\right\}_{n=1}^\iy=
\left\{f_n\right
\}_{n=1}^\iy$, where
$\tau_n\in(2/3,1),\,n\in\N$, and $\lim_{n\to\iy}\tau_n=1$,
there exist a subsequence
$\{f_{0,\tau_{n_d}}\}_{d=1}^\iy$ and a function
$f_0\in B_V\cap L_\iy(\R^m)
=\bigcap_{d=1}^\iy \left(B_{(1/\tau_{n_d})V}
\cap L_\iy(\R^m)\right)$
such that
$\lim_{d\to\iy} f_{0,\tau_{n_d}}=f_0$
uniformly on any compact set in $\CC^m$.

Note that by \eqref{E3.46n} and \eqref{E3.44n},
$f_0$ is a nontrivial function from
$B_{V}\cap L_{p,\prod_{j=1}^m\left\vert y_j\right
\vert^{\al_j}}(\R^m)$. Then using \eqref{E3.47n},
we obtain
\ba
&&\limsup_{a\to\iy}
a^{-\left(m+\sum_{j=1}^m\al_j\right)/p}
 \NN_{0}\left(\PP_{aV},
  L_{p,\prod_{j=1}^m\left\vert t_j\right\vert^{\al_j}
  \left(1-t_j^2\right)^{\be_j}}(Q^m)\right)\nonumber\\
&&\le \lim_{M\to\iy}\lim_{n\to\iy}
\left\vert f_{0,\tau_n}(0)\right\vert/
\left\|f_{0,\tau_n}\right\|_{L_{p,\prod_{j=1}^m\left\vert y_j\right
\vert^{\al_j}}(Q^m(M))}\\
&&=\left\vert f_{0}(0)\right\vert/
\left\|f_{0}\right\|_{L_{p,\prod_{j=1}^m\left\vert y_j\right
\vert^{\al_j}}(\R^m)}\\
&&\le\NN_0\left(B_{V}\cap L_{p,\prod_{j=1}^m\left\vert t_j
\right\vert^{\al_j}}(\R^m),
  L_{p,\prod_{j=1}^m\vert t_j\vert^{\al_j}}(\R^m)\right).
\ea
Thus \eqref{E3.32n} holds true, and this completes
the proof of the theorem.
\hfill$\Box$

\section{The $m$-dimensional Cube $Q^m$.
Proof of Theorem  \ref{T1.2}}\label{S3an}
\noindent
\setcounter{equation}{0}
Recall that $\PP_{aV,e}$ is the set of all
 even polynomials in each variable from
$\PP_{aV}$.
Throughout the section we assume that
$V\subset\R^m$ satisfies the $\Pi$-condition,
$\la_j\ge 0,\,1\le j\le m$,
and $p\in[1,\iy)$.
The proof reduces Theorem \ref{T1.2}
to Theorem \ref{T1.2a} by using
Lemma \ref{L2.7}
 and equalities between
the sharp constants on the cube $Q^m$ discussed
in the lemma below.

\begin{lemma}\label{L4.1}
The following equalities hold true:
\bna\label{E3.1n}
&&\NN_{(1,\ldots,1)}\left(\PP_{aV},
  L_{p,\prod_{j=1}^m\left(1-x_j^2
  \right)^{\la_j-1/2}}(Q^m)\right)\nonumber\\
 && =4^{-(\sum_{j=1}^m\la_j)/p}
  \NN_{0}\left(\PP_{2aV,e},
  L_{p,\prod_{j=1}^m\left\vert t_j
  \right\vert^{2\la_j}\left(1-t_j^2
  \right)^{\la_j-1/2}}(Q^m)\right)\nonumber\\
 && =4^{-(\sum_{j=1}^m\la_j)/p}
  \NN_{0}\left(\PP_{2aV},
  L_{p,\prod_{j=1}^m\left\vert t_j\right
  \vert^{2\la_j}\left(1-t_j^2\right)^{\la_j-1/2}}(Q^m)\right).
\ena
\end{lemma}
\begin{proof}
We first prove that
\beq\label{E3.2n}
\PP_{2aV,e}=S
:=\left\{P_{2a}(t)
=R_a\left(1-2t_1^2,\ldots,1-2t_m^2\right):R_a\in \PP_{aV}\right\}.
\eeq
Indeed, if $k\in aV\cap\Z^m_+$,
that is, $R_a(x):=x^k$ is a monomial from $\PP_{aV}$,
 then the polynomial
$P_{2a}(t):=\left(1-2t_1^2\right)^{k_1}\ldots
\left(1-2t_m^2\right)^{k_m}$ from $S$
 belongs to
$\PP_{\Pi^m(2k),e}\subseteq\PP_{2aV,e}$, by the $\Pi$-condition.
Therefore, $S\subseteq\PP_{2aV,e}$. Conversely,
if $2k\in 2aV\cap\Z^m_+$,
that is, $P_{2a}(t):=t^{2k}$ is a monomial from $\PP_{2aV,e}$,
  then the polynomial
$R_a(x):=\left(\frac{1-x_1}{2}\right)^{k_1}\ldots
\left(\frac{1-x_m}{2}\right)^{k_m} $ belongs to
$\PP_{\Pi^m(k)}\subseteq\PP_{aV}$, by the $\Pi$-condition;
hence $P_{2a}\in S$.
Therefore, $\PP_{2aV,e}\subseteq S$, and \eqref{E3.2n} holds true.

Next, making the substitution
$x=\left(1-2t_1^2,\ldots,1-2t_m^2\right): [0,1]^m\to Q^m$
and setting
$P_{2a}(t):=R_a\left(1-2t_1^2,\ldots,1-2t_m^2\right)$
for any $R_a\in \PP_{aV}$, we have
\bna\label{E3.3n}
&&\frac{\left\vert R_a(1,\ldots,1)\right\vert^p}
{\int_{Q^m}\left\vert R_a(x)\right\vert^p
\prod_{j=1}^m\left(1-x_j^2\right)^{\la_j-1/2}dx}\nonumber\\
&&=\frac{\left\vert P_{2a}(0)\right\vert^p}
{4^{m/2+\sum_{j=1}^m\la_j}
\int_{[0,1]^m}\left\vert P_{2a}(t)\right\vert^p
\prod_{j=1}^m t_j^{2\la_j}
\left(1-t_j^2\right)^{\la_j-1/2}dt}\nonumber\\
&&=\frac{\left\vert P_{2a}(0)\right\vert^p}
{4^{\sum_{j=1}^m\la_j}
\left\|P_{2a}\right\|^p_{ L_{p,\prod_{j=1}^m
\left\vert t_j\right\vert^{2\la_j}
\left(1-t_j^2\right)^{\la_j-1/2}}(Q^m)}}.
\ena
Then the first equality in \eqref{E3.1n}
follows
from \eqref{E3.2n} and \eqref{E3.3n}.
 Furthermore, using the symmetrization trick from Remark \ref{R1.4},
 we see that
  for any $P_{2a}\in \PP_{2aV}$ the polynomial
  \ba
 P_{2a,e}(t):=2^{-m}\sum_{\left\vert \de_j\right\vert=1,\,1\le j\le m}
 P_{2a}\left(\de_1t_1,\ldots, \de_mt_m\right)
 \ea
 belongs to $\PP_{2aV,e}$ and
 \ba
 \frac{\left\vert P_{2a}(0)\right\vert}
 {\left\| P_{2a}\right\|_{L_{p,\prod_{j=1}^m\left\vert t_j\right
  \vert^{2\la_j}\left(1-t_j^2\right)^{\la_j-1/2}}(Q^m)}}
  \le \frac{\left\vert P_{2a,e}(0)\right\vert}
 {\left\| P_{2a,e}\right\|_{L_{p,\prod_{j=1}^m\left\vert t_j\right
  \vert^{2\la_j}\left(1-t_j^2\right)^{\la_j-1/2}}(Q^m)}}.
 \ea
Hence the second equality in \eqref{E3.1n}
holds true.
\end{proof}
\noindent
\emph{Proof of Theorem \ref{T1.2}.}
 An extremal polynomial $P^*$
  for
  $\NN\left(\PP_{aV},
  L_{p,\prod_{j=1}^m\left(1-x_j^2\right)^{\la_j-1/2}}(Q^m)\right)$
   whose uniform norm is attained at
 the vertex
 $x_0=(1,\ldots,1)$
 of $Q^m$, that is,
 $\|P^*\|_{L_\iy(Q^m)}=\vert P^*((1,\ldots,1))\vert$
 exists by Lemma \ref{L2.7}.
Then Lemma \ref{L4.1} implies the equality
 \bna\label{E3.19n}
&&\NN\left(\PP_{aV},
  L_{p,\prod_{j=1}^m\left(1-x_j^2\right)^{\la_j-1/2}}(Q^m)\right)\nonumber\\
 && =4^{-(\sum_{j=1}^m\la_j)/p}
  \NN_{0}\left(\PP_{2aV},
  L_{p,\prod_{j=1}^m\left\vert t_j\right\vert^{2\la_j}
  \left(1-t_j^2\right)^{\la_j-1/2}}(Q^m)\right).
\ena
Next, setting $\al_j=2\la_j,\,\be_j=\la_j-1/2,\,1\le j\le m$,
and replacing $a$ with $2a$,  we obtain from
limit relations \eqref{E1.8a} and \eqref{E1.9a} of Theorem \ref{T1.2a}
\bna \label{E3.20n}
  &&\lim_{a\to\iy}(2a)^{-\left(m+2\sum_{j=1}^m\la_j\right)/p}
  \NN_{0}\left(\PP_{2aV},
  L_{p,\prod_{j=1}^m\left\vert t_j\right\vert^{2\la_j}
  \left(1-t_j^2\right)^{\la_j-1/2}}(Q^m)\right)\nonumber\\
  &&=\NN_0\left(B_{V}\cap L_{p,\prod_{j=1}^m
  \left\vert t_j\right\vert^{2\la_j}}(\R^m),
  L_{p,\prod_{j=1}^m\left\vert t_j
  \right\vert^{2\la_j}}(\R^m)\right)\nonumber\\
  &&= \vert f_0(0)\vert/\|f_0\|_
{L_{p,\prod_{j=1}^m\left\vert t_j\right\vert^{2\la_j}}(\R^m)},
 \ena
 where
 $f_0\in  \left(B_V\cap L_{p,\prod_{j=1}^m\left\vert t_j
 \right\vert^{2\la_j}}(\R^m)\right)\setminus\{0\}$.
 Thus equalities \eqref{E1.8} and \eqref{E1.9}
 of Theorem \ref{T1.2} immediately follow
from relations \eqref{E3.19n} and \eqref{E3.20n}.\hfill$\Box$

\section{The $m$-dimensional Ball $B^m$.
Proofs of Theorems  \ref{T1.3} and \ref {T1.3a}}\label{S4n}
\noindent
\setcounter{equation}{0}
Throughout the section we assume that
$\la\ge 0$
and $p\in[1,\iy)$.
The proofs of Theorems \ref{T1.3}
and \ref {T1.3a} are based on
Lemmas \ref{L2.9} and \ref{L2.11} and on four
propositions of independent interest below.
Three of these propositions discuss equalities between
the sharp constants on $\BB^m$ or $\R^m$.

We first need three \textit{invariance theorems}
(this term was introduced in
\cite{G2008, G2019}).
Let $D(m),\,m\ge 2,$ be the group of all proper and
improper rotations $\rho=\rho_m$ (about the origin)
of $\R^m$. We identify $D(m)$ with the group $D^*(m)$ of all
$m\times m$ orthogonal
matrices $A=A_m$
that is isomorphic to $D(m)$
since $\rho\in D(m)$ if and only if $ \rho x=A(\rho)x^T$,
where $A(\rho)$ is an
$m\times m$ orthogonal matrix with $\vert\det A(\rho)\vert=1,\,
x\in\R^m$,
and $x^T$ is a column vector.

Let $x_0=x_0(m):=(1,0,\ldots,0)\in\R^m$, and let ${x}_0^T$ be
the column version of $x_0$.
Next, let $D(m,x_0),\,m\ge 2,$ be the subgroup
 of $D(m)$ of all proper and
improper rotations $\rho$ around the $x_1$-axis
of $\R^m$ (i.e., $\rho x_0=x_0$)
 that is isomorphic to the group $D^*(m,x_0)$
 of all $m\times m$ orthogonal
matrices $A_m$, satisfying the condition $A_mx_0^T=x_0^T$.
The group $D^*(m,x_0)$ can be characterized as
the subgroup of $D^*(m)$ of all
 $m\times m$ orthogonal matrices $A_m$ of the
 following block form:
 \beq\label{E4.1aa}
 A_m=\begin{bmatrix}
 1 &0\\
 0& A_{m-1}
 \end{bmatrix},
 \eeq
 where $A_{m-1}\in D^*(m-1)$.
 Representation \eqref{E4.1aa} immediately
  follows from the condition $A_mx_0^T=x_0^T$ and
  from the fact that the transpose  of
    $A_m\in  D^*(m,x_0)$ belongs to $D^*(m,x_0)$ as well.
For example, $D^*(2,x_0)=\{I_2, A_2\}$, where $I_2$ is the
 $2 \times 2$ identity matrix and
 \ba
 A_2=\begin{bmatrix}
 1 &0\\
 0& -1
 \end{bmatrix}.
 \ea
So the only nontrivial rotation from $D(2,x_0)$ is the reflection around
the $x_1$-axis.

In terms of rotations, representation \eqref{E4.1aa} of
$A_m\in  D^*(m,x_0)$ is equivalent to the criterion
\beq\label{E4.1ab}
\rho_m\in D(m,x_0) \Longleftrightarrow \rho_m x=(x_1,\rho_{m-1} x^\prime),
\qquad x\in\R^m,
\eeq
where $x^\prime:=(x_2,\ldots,x_m)$ and $\rho_{m-1}\in D(m-1)$
is a rotation about the origin of the $(m-1)$-dimensional subspace
 $\{x\in\R^m:x_1=0\}$ of $\R^m$.

We say that $f:\R^m\to\R^m$ is \textit{invariant} under
a subgroup $D$ of $D(m)$
 if $f(\rho x)=f(x)$ for all $\rho\in D$ and $x\in\R^m$.
 The set of all polynomials $P\in\PP_{n,m}$ that are invariant under
$D$ is denoted by $\PP_{n,m}^D$,
and the set of all entire functions
 $f\in B_{\BB^m}$ that are invariant under
$D$ is denoted by $B_{\BB^m}^D$.

In addition,
 let $\PP_{n,2,e}$ be the set of all polynomials
 $P_2(u,v)$ of two variables from $\PP_{n,2}$
  that are even with respect to the second variable.
  We first find the representation of a polynomial that is
  invariant under $D(m,x_0)$.
  Note that the proof of the following proposition is
  the refinement of the proof of Lemma 4.2.12 in \cite{SW1971}.

 \begin{proposition}\label{P4.1}
A polynomial $P$ belongs to $\PP_{n,m}^{D(m,x_0)},\,m\ge 2$,
if and only if
there exists a polynomial  $P_2\in\PP_{n,2,e}$
such that
$P(x)=P_2\left(x_1,\left(\sum_{j=2}^mx_j^2\right)^{1/2}\right)$.
\end{proposition}
\begin{proof}
 It follows from \eqref{E4.1ab} that for any
 $P_2\in \PP_{n,2,e}$ the  polynomial
$P(x)=P_2\left(x_1,\left(\sum_{j=2}^mx_j^2\right)^{1/2}\right)$
 is invariant under $D(m,x_0)$.

 Conversely, let a polynomial
 $P(x)=\sum_{l=0}^nx_1^{l}P_{l,n}(x_2,\ldots,x_m)$
 from $\PP_{n,m}$
 be invariant under $D(m,x_0)$.
 Here, $P_{l,n}\in \PP_{n-l,m-1},\,0\le l\le n$.
 Then the polynomials
 $\frac{1}{l!}\frac{\partial^l P(x)}{\partial x_1^l},\,
 0\le l\le n,$ are
 invariant under $D(m,x_0)$ as well
 because by \eqref{E4.1ab},
 $Q(x_1+\de,\rho_{m-1} x^\prime)=Q(x_1+\de,x^\prime)$
 for any $\de\in\R^1$ and every
 $Q\in \PP_{n-\nu,m}^{D(m,x_0)},\,0\le \nu\le n$.
 Therefore, the polynomials
 \ba
 P_{l,n}(x^\prime)=P_{l,n}(x_2,\ldots,x_m)
 =\left.\frac{1}{l!}\frac{\partial^l P(x)}{\partial x_1^l}
  \right\vert_{x_1=0},\qquad 0\le l\le n,
 \ea
 are invariant under $D(m,x_0)$.
 Hence taking into account that $P_{l,n}$
 are independent of $x_1$, we conclude
  from \eqref{E4.1ab} that
 $P_{l,n}(\rho_{m-1} x^\prime)
 =P_{l,n}(x^\prime),\,0\le l\le n$,
 for every rotation
 $\rho_{m-1}$ of the
 subspace $\{x\in\R^m:x_1=0\}$ of $\R^m$.
 Therefore, $P_{l,n}$ are invariant under $D(m-1)$
 and by Lemma 4.2.11 from \cite{SW1971},
 \ba
 P_{l,n}(x_2,\ldots,x_m)
 =\sum_{j=0}^{\lfloor (n-l)/2\rfloor}b_{j,l}
 (x_2^2+\ldots+x_m^2)^j,
 \qquad 0\le l\le n.
 \ea
 Thus
 \ba
 P(x)=\sum_{0\le l+2\nu\le
 n}c_{l,\nu}x_1^l(x_2^2+\ldots+x_m^2)^\nu
 =P_2\left(x_1,\left(\sum_{j=2}^mx_j^2\right)^{1/2}\right),
 \ea
 where $P_2\in \PP_{n,2,e}$, and the proposition is established.
 \end{proof}
 \noindent
 A different version
  of Proposition \ref{P4.1} was discussed in \cite[Proposition 5.1]{G2008}.

  In the proofs of the next two propositions, we need the following
  simple observation: if $W(x)=\psi(\vert x\vert)$ is a radial weight,
  then for $\Omega=\BB^m$ or $\Omega=\R^m$ the norm in
  $L_{p,W}(\Omega)$ is invariant under rotation, that is,
  for any
  $F\in L_{p,W}\left(\R^m\right)$ and any
  rotation $\rho \in D(m)$,
  \ba
  \|F(\rho\cdot)\|_{L_{p,W}\left(\R^m\right)}
  =\|F(\cdot)\|_{L_{p,W}\left(\R^m\right)},\qquad p\in[1,\iy).
  \ea

Next, we reduce the sharp constant for the $m$-dimensional ball
  $\BB^m,\,m\ge 2$,
  to the one on the disk $\BB^2$.
 Recall that $x_0=x_0(m)=(1,0,\ldots,0)\in\R^m$.

  \begin{proposition}\label{P4.2}
  The following equalities hold true for $m\ge 2$:
  \bna\label{E4.1}
   &&\NN_{x_0(m)}\left(\PP_{n,m},
  L_{p,(1-\vert x\vert^2)^{\la-1/2}}\left(\BB^m\right)\right)\nonumber\\
  &=&\NN_{x_0(m)}\left(\PP_{n,m}^{D(m,x_0(m))},
  L_{p,(1-\vert x\vert^2)^{\la-1/2}}\left(\BB^m\right)\right)\nonumber\\
  &=&C_{17}\,\NN_{x_0(2)}\left(\PP_{n,2,e},
  L_{p,(1-u^2-v^2)^{\la-1/2}\vert v\vert^{m-2}}\left(\BB^2\right)\right),
  \ena
  where
  \beq\label{E4.2}
  C_{17}=C_{17}(m,p)
  :=\left(\frac{\Gamma((m-1)/2)}{\pi^{(m-1)/2}}\right)^{1/p}.
  \eeq
  \end{proposition}
  \begin{proof}
  The proof of the first equality in \eqref{E4.1}
  follows general ideas developed in \cite[Theorems 2.1 and 2.2]{G2019}.
  Given $P\in\PP_{n,m}$, we define the Haar integral
  \beq\label{E4.3}
P^*(x):=\int_{D(m,x_0)}P(\rho x)d\mu(\rho),
\eeq
where $\mu$ is the invariant Haar measure on $D(m,x_0)$ (see, e.g.,
\cite[Theorem 5.14]{R1973}).
Since $\rho:\BB^m\to \BB^m$ is a linear map,
 $P(\rho\cdot)\in\PP_{n,m},\,
\rho\in D(m,x_0)$, and $P^*\in\PP_{n,m}$.
In addition, for every $\rho^*\in D(m,x_0)$,
\ba
P^*(\rho^*x)=\int_{D(m,x_0)}P(\rho\rho^*x)d\mu(\rho)
=\int_{D(m,x_0)}P(\rho x)d\mu(\rho)
=P^*(x).
\ea
Therefore, $P^*\in \PP_{n,m}^{D(m,x_0)}$.
Using now the generalized
  Minkowski inequality (see, e.g., \cite[Lemma 3.2.15]{DS1988})
  and taking into account the fact that the norm
  $\|\cdot\|_{ L_{p,\left(1-\vert x\vert^2\right)^{\la-1/2}}(\BB^m)}$
  is invariant under rotation, we obtain
\beq\label{E4.4}
\left\|P^*\right\|_{ L_{p,\left(1-\vert x\vert^2\right)^{\la-1/2}}(\BB^m)}
\le \int_{D(m,x_0)}\|P(\rho\cdot)\|_
{ L_{p,\left(1-\vert x\vert^2\right)^{\la-1/2}}(\BB^m)}d\mu(\rho)
= \left\|P\right\|_
{ L_{p,\left(1-\vert x\vert^2\right)^{\la-1/2}}(\BB^m)}.
\eeq
Furthermore, since $\rho\in D(m,x_0)$ in \eqref{E4.3}, we have
\beq\label{E4.5}
P^*(x_0)=P(x_0).
\eeq
Therefore, by \eqref{E4.4} and \eqref{E4.5},
\ba
\frac{\left\vert P(x_0)\right\vert}
{\left\|P\right\|_
{ L_{p,\left(1-\vert x\vert^2\right)^{\la-1/2}}(\BB^m)}}
\le \frac{\left\vert P^*(x_0)\right\vert}
{\left\|P^*\right\|_
{ L_{p,\left(1-\vert x\vert^2\right)^{\la-1/2}}(\BB^m)}},
\ea
which proves the first equality in \eqref{E4.1}.

Next, by Proposition \ref{P4.1},
\bna\label{E4.6}
&&\NN_{x_0(m)}\left(\PP_{n,m}^{D(m,x_0(m))},
  L_{p,(1-\vert x\vert^2)^{\la-1/2}}\left(\BB^m\right)\right)\nonumber\\
  &&=\sup_{P_2\in\PP_{n,2,e}\setminus\{0\}}
  \frac{\left\vert P_2(1,0)\right\vert}
{\left(\bigint_{\BB^m}\left\vert P_2\left(x_1,\left(\sum_{j=2}^mx_j^2\right)^{1/2}
\right)\right\vert^p(1-\vert x\vert^2)^{\la-1/2}dx\right)^{1/p}}.
  \ena
Using the spherical coordinate system in $\R^m$,
we see that $x_1=r\cos \vphi$
and  $\left(\sum_{j=2}^mx_j^2\right)^{1/2}=r\vert\sin \vphi\vert$,
where $\vphi\in[0,2\pi)$ if $m=2$ and $\vphi\in[0,\pi]$ if $m>2$.
Since $P_2\in\PP_{n,2,e}$ is even with respect to the second variable,
we obtain
\bna\label{E4.7}
&&\bigint_{\BB^m}\left\vert P_2\left(x_1,\left(\sum_{j=2}^mx_j^2\right)^{1/2}
\right)\right\vert^p(1-\vert x\vert^2)^{\la-1/2}dx\nonumber\\
&&=(1/C_{17})^p
\int_{0}^1\int_{0}^{2\pi}\left\vert P_2\left(r\cos \vphi,r\sin \vphi
\right)\right\vert^p(1-r^2)^{\la-1/2}r^{m-1}\vert\sin \vphi\vert^{m-2}d\vphi\,dr
\nonumber\\
&&=(1/C_{17})^p
\int_{(u,v)\in \BB^2}\left\vert P_2\left(u,v
\right)\right\vert^p(1-u^2-v^2)^{\la-1/2}\vert v\vert^{m-2}du\,dv,
\ena
where
\ba
C_{17}=
\left(\frac{2\int_0^\pi \sin^{m-2}\vphi\,d\vphi}
{\left\vert S^{m-1}\right\vert_{m-1}}\right)^{1/p}
=\left(\frac{\Gamma((m-1)/2)}{\pi^{(m-1)/2}}\right)^{1/p}.
\ea
Thus the second equality in \eqref{E4.1} follows from \eqref{E4.6}
and \eqref{E4.7}.
\end{proof}
In the next proposition, we reduce the multivariate sharp constant
for $\R^m$ to the univariate one for $\R^1$.

  \begin{proposition}\label{P4.2a}
  The following equalities hold true for $m\ge 1$:
  \bna\label{E4.1a}
   &&\NN_{0}\left(B_{\BB^m}\cap
   L_{p,\vert t\vert^{2\la}}\left(\R^m\right),
   L_{p,\vert t\vert^{2\la}}\left(\R^m\right)\right)\nonumber\\
  &=&\NN_{0}\left(B_{\BB^m}^{D(m)}\cap
   L_{p,\vert t\vert^{2\la}}\left(\R^m\right),
   L_{p,\vert t\vert^{2\la}}\left(\R^m\right)\right)\nonumber\\
  &=&C_{18}\,\NN_0\left(B_{1,e}\cap L_{p,\vert u\vert^{m+2\la-1}}(\R^1),
  L_{p,\vert u\vert^{m+2\la-1}}(\R^1)\right),
  \ena
  where
  \beq\label{E4.2a}
  C_{18}=C_{18}(m,p)
  :=\left(\frac{2}{\left\vert S^{m-1}
  \right\vert_{m-1}}\right)^{1/p}
  =\left(\frac{\Gamma(m/2)}{\pi^{m/2}}\right)^{1/p}.
  \eeq
  \end{proposition}
  \begin{proof}
  Equalities \eqref{E4.1a} for $\la=0$ were proved in
  \cite[Corollary 3.11]{G2019}. The proof was based on
  \cite[Theorem 2.2]{G2019} and several propositions
  from \cite{G2019}, and it
  was comparatively long. The proof of Proposition
  \ref{P4.2a} can be copied from the aforementioned one
  if we take into account the fact that the norm in
  $L_{p,\vert x\vert^{2\la}}\left(\R^m\right)$ is invariant
  under rotation.
  \end{proof}
Next, we reduce the multivariate and bivariate sharp constants in
\eqref{E4.1} to the univariate one.

\begin{proposition}\label{P4.3}
The following equality holds true for $m\ge 1$:
  \beq\label{E4.8}
   \NN_{x_0(m)}\left(\PP_{n,m},
  L_{p,(1-\vert x\vert^2)^{\la-1/2}}\left(\BB^m\right)\right)
  =C_{19}\,\NN_{1}\left(\PP_{n},
  L_{p,(1-u^2)^{m/2+\la-1}}([-1,1])\right),
  \eeq
  where
  \beq\label{E4.9}
  C_{19}=C_{19}(m,p,\la):=\left(\frac{\Gamma(\la+m/2)}
  {\pi^{(m-1)/2}\Gamma(\la+1/2)}\right)^{1/p}.
  \eeq
\end{proposition}
\begin{proof}
Since \eqref{E4.8} is trivial for $m=1$, we assume that $m\ge 2$.
Then using Proposition \ref{P4.2} and making
the substitution $v=\tau\,\sqrt{1-u^2},\,\tau\in[-1,1]$,
we obtain
\bna\label{E4.10}
&&\NN_{x_0(m)}^p\left(\PP_{n,m},
  L_{p,(1-\vert x\vert^2)^{\la-1/2}}\left(\BB^m\right)\right)\nonumber\\
  &&=C_{17}^p\sup_{P_2\in\PP_{n,2,e}\setminus\{0\}}
 \frac{\left\vert P_2(1,0)\right\vert^p}
 {\int_{-1}^1\left(\int_{-\sqrt{1-u^2}}^{\sqrt{1-u^2} } \left\vert P_2\left(u,v
\right)\right\vert^p(1-u^2-v^2)^{\la-1/2}\vert v\vert^{m-2}dv\right)du}\nonumber\\
&&= C_{17}^p
\sup_{P_2\in\PP_{n,2,e}\setminus\{0\}}\nonumber\\
&&\frac{\left\vert P_2(1,0)\right\vert^p}
 {\int_{-1}^1\left(\int_{-1}^1 \left\vert P_2\left(u,\tau\,\sqrt{1-u^2}\right)
\right\vert^p(1-u^2)^{m/2+\la-1}du\right)\tau^{m-2}(1-\tau^2)^{\la-1/2}d\tau}.
\ena
Since  the function
\ba
Q_\tau(u):=P_2\left(u,\tau\,\sqrt{1-u^2}\right),\qquad
\tau\in[-1,1],
\quad P_2\in\PP_{n,2,e}\setminus\{0\},
\ea
 is a polynomial in $u$ of
degree at most $n$ and $Q_\tau(1)=P_2(1,0)$,
we have for each fixed $\tau\in [-1,1]$ and every fixed
$P_2\in\PP_{n,2,e}\setminus\{0\}$,
\bna\label{E4.11}
&&\int_{-1}^1 \left\vert P_2\left(u,\tau\,\sqrt{1-u^2}\right)
\right\vert^p(1-u^2)^{m/2+\la-1}du\nonumber\\
&&\ge \left\vert P_2(1,0)\right\vert^p
\inf_{Q\in\PP_{n}\setminus\{0\}}
\frac{\int_{-1}^1 \left\vert Q(u)\right\vert^p
(1-u^2)^{m/2+\la-1}du}
{\left\vert Q(1)\right\vert^p}.
\ena
Then combining \eqref{E4.10} and \eqref{E4.11}, we obtain
\bna\label{E4.12}
&&\NN_{x_0(m)}^p\left(\PP_{n,m},
  L_{p,(1-\vert x\vert^2)^{\la-1/2}}\left(\BB^m\right)\right)\nonumber\\
 && \le C_{17}^p
  \left(\int_{-1}^1 \tau^{m-2}(1-\tau^2)^{\la-1/2}d\tau\right)^{-1}
  \NN_{1}^p\left(\PP_{n},
  L_{p,(1-u^2)^{m/2+\la-1}}([-1,1])\right)\nonumber\\
  &&= C_{19}^p \NN_{1}^p\left(\PP_{n},
  L_{p,(1-u^2)^{m/2+\la-1}}([-1,1])\right),
  \ena
  where $C_{17}$ and $C_{19}$ are defined by
  \eqref{E4.2} and \eqref{E4.9}, respectively.
  On the other hand, repeating calculations \eqref{E4.10}
  for polynomials $P_2(u,v)=Q(u)$ that are independent of $v$,
  we have
\bna\label{E4.13}
&&\NN_{x_0(m)}^p\left(\PP_{n,m},
  L_{p,(1-\vert x\vert^2)^{\la-1/2}}\left(\BB^m\right)\right)\nonumber\\
  &&\ge C_{17}^p\sup_{Q\in\PP_{n}\setminus\{0\}}
 \frac{\left\vert Q(1)\right\vert^p}
 {\int_{-1}^1\left(\int_{-\sqrt{1-u^2}}^{\sqrt{1-u^2} } \left\vert Q(u)
 \right\vert^p(1-u^2-v^2)^{\la-1/2}\vert v\vert^{m-2}dv\right)du}\nonumber\\
 &&=C_{19}^p \NN_{1}^p\left(\PP_{n},
  L_{p,(1-u^2)^{m/2+\la-1}}([-1,1])\right).
 \ena
 Thus \eqref{E4.8} follows from \eqref{E4.12} and \eqref{E4.13}.
 \end{proof}
 \noindent
 \textit{Proof of Theorem \ref{T1.3}.}
 Recall that the constants $A_1$ and $A_2$ are defined by
 \eqref{E1.10a} and \eqref{E1.10ab}, respectively.

 Next, by Lemma \ref{L2.9} for $\la=0$ and
 Lemma \ref{L2.11}  for $\la>0$,
 there exists an extremal polynomial $P^*$
  for $\NN\left(\PP_{n,m},
  L_{p,(1-\vert x\vert^2)^{\la-1/2}}\left(\BB^m\right)\right),
  \,p\in[1,\iy),\,\la\ge 0,$
 and there exists
 $x_0\in S^{m-1}$ such that
 $\|P^*\|_{L_\iy(\BB^m)}=\vert P^*(x_0)\vert$.

 Without loss of generality we can assume that
 $x_0=x_0(m)=(1,\,0,\ldots,0)$ for $m\ge 2$.
 Indeed, if $x_0\neq x_0(m)$, then
 there exists the rotation $\rho_0\in D(m)$ such that
 $\rho_0 x_0=x_0(m)$. Then the polynomial
 $P^{**}(x):=P^{*}(\rho_0 x)$ is an extremal polynomial
  for $\NN\left(\PP_{n,m},
  L_{p,(1-\vert x\vert^2)^{\la-1/2}}\left(\BB^m\right)\right)$ and
  $\|P^{**}\|_{L_\iy(\BB^m)}=\vert P^{**}(1,\,0,\ldots,0)\vert$.

  Therefore, for $p\in[1,\iy),\,m\ge 1$, and $\la\ge 0,$
  \beq\label{E4.14}
  \NN\left(\PP_{n,m},
  L_{p,(1-\vert x\vert^2)^{\la-1/2}}\left(\BB^m\right)\right)
  =\NN_{x_0(m)}\left(\PP_{n,m},
  L_{p,(1-\vert x\vert^2)^{\la-1/2}}\left(\BB^m\right)\right).
  \eeq
  Note that for $m=1$ \eqref{E4.14} follows from \eqref{E1.6d}.
  Thus the first relation in \eqref{E1.10} of Theorem \ref{T1.3}
  with $A_1=2^{1/p}C_{19}$ in \eqref{E1.10a},
  where $C_{19}$ is defined by \eqref{E4.9},
  follows from equalities \eqref{E4.14} and \eqref{E4.8}
  and limit relation \eqref{E1.6c}.
  The second equality in \eqref{E1.10} of Theorem \ref{T1.3}
  with $A_2=A_1/C_{18}$ in \eqref{E1.10ab},
  where $C_{18}$ is defined by \eqref{E4.2a},
  is an immediate consequence of Proposition \ref{P4.2a}.

  Finally, to prove \eqref{E1.10b}, we choose
  $f_0(t):=g_{0}(\vert t\vert)$, where
  $g_{0}
  \in  \left(B_{1,e}\cap L_{p,\vert u
 \vert^{m+2\la-1}}(\R^1)\right)\setminus\{0\}$ and
 \beq \label{E4.15}
  \NN_0\left(B_{1,e}\cap L_{p,\vert u\vert^{2\la}}(\R^1),
  L_{p,\vert u\vert^{m+2\la-1}}(\R^1)\right)
= \vert g_{0}(0)\vert/\|g_{0}
\|_{ L_{p,\vert u\vert^{m+2\la-1}}(\R^1)}.
 \eeq
 The existence of $g_0$ was proved in
 \cite[Theorem 4.3]{G2019} (see also relations \eqref{E1.6c}
  and \eqref{E1.6ca}
 with $\la\ge 0$, replaced by  $\la+(m-1)/2$).
 Then $f_{0}
  \in  \left(B_{\BB^m}\cap L_{p,\vert t
 \vert^{2\la}}(\R^m)\right)\setminus\{0\}$
 and by \eqref{E4.15} and by
  Proposition \ref{P4.2a},
 \beq\label{E4.16}
 \NN_0\left(B_{\BB^m}\cap L_{p,\vert t\vert^{2\la}}(\R^m),
  L_{p,\vert t\vert^{2\la}}(\R^m)\right)
  =\vert f_0(0)\vert/\|f_0\|_{ L_{p,\vert t\vert^{2\la}}(\R^m)}.
  \eeq
  Thus \eqref{E1.9a} follows from \eqref{E1.9} and \eqref{E4.16}.
  This completes the proof of the theorem.
  \hfill$\Box$\vspace{.12in}\\
  \noindent
 \textit{Proof of Theorem \ref{T1.3a}.}
 Using Proposition \ref{P4.3} and formula \eqref{E1.6ba}
 with $\la$ replaced by $\la+(m-1)/2$, we arrive at
 \eqref{E1.10c}.
 Next, the asymptotic
  \bna\label{E4.17}
  &&\NN\left(\PP_{n,m},
  L_{2,\left(1-\vert x\vert^2\right)^{\la-1/2}}
  \left(\BB^m\right)\right)\nonumber\\
  &&=\frac{n^{\la+m/2}(1+o(1))}
  {\left(2^{2\la+m-2}\pi^{(m-1)/2}(2\la+m)^{1/2}
  \Gamma(\la+1/2)\Gamma(\la+m/2)\right)^{1/2}},\qquad n\to\iy,
  \ena
  immediately follows from \eqref{E1.10c}.
  Therefore, formula \eqref{E1.10d} is a direct consequence of
  \eqref{E1.10} and \eqref{E4.17}.
  \hfill$\Box$\vspace{.12in}\\
  \textbf{Acknowledgements.} We are grateful to both anonymous referees
 for valuable suggestions.

\end{document}